\documentclass[final]{siamltex}

\usepackage{amsmath,amsfonts,amssymb}
\usepackage[usenames]{color}

\usepackage[arXiv]{optional}

\usepackage[mathscr]{eucal}

\usepackage{amsbsy}

\usepackage{bm}

\usepackage{xspace}

\usepackage{paralist}

\usepackage{graphicx}
\DeclareGraphicsExtensions{.png}

\usepackage{subfig}

\usepackage[colorlinks,urlcolor=blue,citecolor=blue,linkcolor=blue]{hyperref}

\usepackage[notref,notcite]{showkeys}

\usepackage{algorithm}
\usepackage{algpseudocode}

\newtheorem{property}[theorem]{Property}

\newtheorem{Example}[theorem]{Example}
\newenvironment{example}{\begin{Example}\rm}{~~$\square$\end{Example}}
\newenvironment{example2}[1]{\begin{Example}[#1]\rm}{~~$\square$\end{Example}}

\newcommand{\Real}{\mathbb{R}}

\newcommand{\Complex}{\mathbb{C}}

\newcommand{\Tra}{^T} 
\newcommand{\Cct}{^\dagger} 
\newcommand{\V}[1]{{\bm{\mathbf{\MakeLowercase{#1}}}}} 
\newcommand{\VE}[2]{\MakeLowercase{#1}_{#2}} 
\newcommand{\M}[1]{{\bm{\mathbf{\MakeUppercase{#1}}}}} 
\newcommand{\T}[1]{\boldsymbol{\mathscr{\MakeUppercase{#1}}}} 
\newcommand{\TE}[2]{\MakeLowercase{#1}_{#2}} 
\newcommand{\qtext}[1]{\quad\text{#1}\quad} 

\newcommand{\TA}{\T{A}}
\newcommand{\Vx}{\V{x}}
\newcommand{\Vy}{\V{y}}
\newcommand{\UB}{\Gamma} 
\newcommand{\US}{\Sigma} 
\newcommand{\Rn}{\Real^n}
\newcommand{\RT}[2]{\Real^{[#1,#2]}} 
\newcommand{\Perm}[1]{\Pi_{#1}}

\newenvironment{inlinemath}{$}{$}

\newcommand{\Sec}[1]{\hyperref[sec:#1]{\S\ref*{sec:#1}}} 
\newcommand{\App}[1]{\hyperref[sec:#1]{Appendix~\ref*{sec:#1}}} 
\newcommand{\Eqn}[1]{\hyperref[eq:#1]{(\ref*{eq:#1})}} 
\newcommand{\Fig}[1]{\hyperref[fig:#1]{Figure~\ref*{fig:#1}}} 
\newcommand{\Tab}[1]{\hyperref[tab:#1]{Table~\ref*{tab:#1}}} 
\newcommand{\Thm}[1]{\hyperref[thm:#1]{Theorem~\ref*{thm:#1}}} 
\newcommand{\Lem}[1]{\hyperref[lem:#1]{Lemma~\ref*{lem:#1}}} 
\newcommand{\Prop}[1]{\hyperref[prop:#1]{Property~\ref*{prop:#1}}} 
\newcommand{\Cor}[1]{\hyperref[cor:#1]{Corollary~\ref*{cor:#1}}} 
\newcommand{\Def}[1]{\hyperref[def:#1]{Definition~\ref*{def:#1}}} 
\newcommand{\Alg}[1]{\hyperref[alg:#1]{Algorithm~\ref*{alg:#1}}} 
\newcommand{\Ex}[1]{\hyperref[ex:#1]{Example~\ref*{ex:#1}}} 

\begin{document}
\title{Shifted Power Method for Computing Tensor Eigenpairs%
  \thanks{This work was funded by the applied mathematics program at the U.S.
    Department of Energy and by an Excellence Award from the Laboratory
    Directed Research \& Development (LDRD) program at Sandia National
    Laboratories. Sandia National Laboratories is a multiprogram
    laboratory operated by Sandia Corporation, a wholly owned subsidiary of Lockheed Martin Corporation,
    for the United States Department of Energy's National Nuclear Security
    Administration under contract DE-AC04-94AL85000.}}
\author{Tamara G. Kolda\footnotemark[2] \and Jackson R. Mayo\footnotemark[2]} 
\maketitle

\opt{draft}{
\centerline {\sc \today}
}

\renewcommand{\thefootnote}{\fnsymbol{footnote}}
\footnotetext[2]{Sandia National Laboratories, Livermore, CA. Email: \{tgkolda,jmayo\}@sandia.gov.}
\renewcommand{\thefootnote}{\arabic{footnote}}

\begin{abstract}
  Recent work on eigenvalues and eigenvectors for tensors of order $m \ge 3$
  has been motivated by
  applications in blind source separation, magnetic resonance imaging,
  molecular conformation, and more. In this paper, we consider methods
  for computing real symmetric-tensor eigenpairs of the form $\TA\Vx^{m-1} =
  \lambda \Vx$ subject to $\|\Vx\|=1$, which is closely related to
  optimal rank-1 approximation of a symmetric
  tensor.  Our contribution is a shifted symmetric higher-order
  power method (SS-HOPM), which we show is guaranteed to converge to a
  tensor eigenpair. SS-HOPM can be viewed as a generalization of the
  power iteration method for matrices or of the symmetric
  higher-order power method.  Additionally, using fixed point
  analysis, we can characterize exactly which eigenpairs can and
  cannot be found by the method. Numerical examples are presented,
  including examples from an extension of the method to finding
  complex eigenpairs.
\end{abstract}

\opt{siam}{
}

\pagestyle{myheadings}
\thispagestyle{plain}
\opt{draft}{
\markboth{\sc Draft --- \today}{\sc Draft --- \today}
}
\opt{arXiv,siam}{
\markboth{\sc T.~G.~Kolda and J.~R.~Mayo}{\sc Shifted Power Method for Computing Tensor Eigenpairs}
}

\section{Introduction}
\label{sec:introduction}

Tensor eigenvalues and eigenvectors have received much attention lately in the
literature \cite{Li05,Qi05,QiSuWa07,Qi07,ChPeZh08,NgQiZh09,WaQiZh09}. The tensor eigenproblem is
important because it has applications in blind source separation
\cite{KoRe02}, magnetic resonance imaging \cite{ScSe08,QiWaWu08},
molecular conformation \cite{Di88}, etc.  There is more than one possible definition for a
tensor eigenpair \cite{Qi05}; in this paper, we
specifically use the following definition.
\begin{definition}
  \label{def:EVP}
  Assume that $\TA$ is a symmetric $m^\text{th}$-order $n$-dimensional
  real-valued tensor. For any $n$-dimensional vector $\Vx$, define
  \begin{equation}
    \label{eq:Axm1}
    \left( \TA\Vx^{m-1} \right)_{i_1} \equiv 
    \sum_{i_2=1}^n \cdots \sum_{i_m=1}^n
    \TE{a}{i_1 i_2 \cdots i_m} \VE{x}{i_2} \cdots \VE{x}{i_m}
    \qtext{for}
    i_1 = 1,\dots,n.
  \end{equation}
  Then $\lambda \in \Real$ is an
  \emph{eigenvalue} of $\TA$ if there exists $\Vx \in \Real^n$ such that
  \begin{equation}\label{eq:EVP}
    \TA\Vx^{m-1} = \lambda \Vx  \qtext{and} \Vx\Tra\Vx=1.
  \end{equation}%
  The vector $\Vx$ is a corresponding \emph{eigenvector}, and
  $(\lambda,\Vx)$ is called an \emph{eigenpair}. 
\end{definition}

\Def{EVP} is equivalent to the Z-eigenpairs defined by Qi
\cite{Qi05,Qi07} and the $l^2$-eigenpairs defined by Lim
\cite{Li05}. In particular, Lim \cite{Li05} observes that any
eigenpair $(\lambda,\Vx)$ is a Karush-Kuhn-Tucker (KKT) point 
(i.e., a constrained stationary point) of the nonlinear optimization
problem 
\begin{equation}\label{eq:NLP}
  \max_{\Vx \in \Rn}
  \TA\Vx^m
    \qtext{subject to}
    \Vx\Tra\Vx = 1,
  \qtext{where}
  \TA\Vx^m
\equiv \sum_{i_1=1}^n \cdots \sum_{i_m=1}^n
    \TE{a}{i_1 \cdots i_m} \VE{x}{i_1} \cdots \VE{x}{i_m}.
\end{equation}
This is equivalent to the problem of finding the best \emph{symmetric} rank-1
approximation of a symmetric tensor \cite{DeDeVa00a}.
We present the more general definition that incorporates complex-valued eigenpairs in \Sec{complex}.

In this paper, we build upon foundational work by Kofidis and Regalia
\cite{KoRe02} for solving \Eqn{NLP}. Their paper is extremely important 
for computing tensor eigenvalues even though it predates
the definition of the eigenvalue problem by three years. 
Kofidis and Regalia consider the higher-order power method
(HOPM) \cite{DeDeVa00a}, a well-known technique for approximation of
higher-order tensors, and show that its
symmetric generalization (S-HOPM) is not guaranteed to
converge. They go on, however, to use convexity theory to
provide theoretical 
results (as well as practical examples) explaining conditions under
which the method is convergent for even-order tensors (i.e., $m$
even). Further, these 
conditions are shown to hold for many problems of practical interest.

In the context of independent component analysis (ICA), both Regalia and Kofidis \cite{ReKo03} and 
Erdogen \cite{Er09} have developed shifted variants of the power 
method and shown that they are monotonically convergent.
We present a similar method in the context of finding real-valued tensor eigenpairs,
called the shifted symmetric higher-order power method (SS-HOPM),
along with theory showing that it is guaranteed to converge to
a constrained stationary point of \Eqn{NLP}. The proof is general and works for both odd-
and even-order tensors (i.e., all $m \ge 3$).  The effectiveness
of SS-HOPM is demonstrated on several examples, including a problem
noted previously \cite{KoRe02} for which S-HOPM does not converge.
We also present a version of SS-HOPM for finding complex-valued tensor
eigenpairs and provide examples of its effectiveness.

As mentioned, there is more than one definition of a tensor eigenpair.
In the case of the 
\emph{$l^m$-eigenpair} (we use $m$ for the tensor order instead of $k$ as in some references)
or \emph{H-eigenpair}, the eigenvalue equation becomes $\TA\Vx^{m-1} = \lambda
\Vx^{[m-1]}$, where $\Vx^{[m-1]}$ denotes the vector $\Vx$ with each
element raised to the $(m-1)^{\text{st}}$ power \cite{Li05,Qi05}.  In this context, Qi,
Wang, and Wang \cite{QiWaWa07} propose some methods specific to
third-order tensors ($m = 3$).  Unlike the ($l^2$-)eigenvalues we consider
here, it is possible to guarantee convergence to the \emph{largest}
$l^m$-eigenvalue for certain classes of nonnegative tensors. For
example, see the power methods proposed by Ng, Qi, and Zhou
\cite{NgQiZh09} and Liu, Zhou, and Ibrahim \cite{LiZhIb10}, the latter
of which also uses a shift to guarantee convergence for any
irreducible nonnegative tensor.

\section{Preliminaries}

Throughout, let $\UB$ and $\US$ denote the
unit ball and sphere on $\Rn$, i.e.,
\begin{displaymath}
  \UB = \{ \Vx \in \Rn : \| \Vx \| \leq 1 \}
  \qtext{and}
  \US = \{ \Vx \in \Rn : \| \Vx \| = 1 \}.
\end{displaymath}
Additionally, define
\begin{displaymath}
  \Perm{m} \equiv \text{the set of all permutations of } (1,\dots,m).
\end{displaymath}
Let $\Vx \bot \Vy$ denote $\Vx\Tra\Vy = 0$, and define
$\Vx^{\bot} \equiv \{ \Vy \in \Rn : \Vx \bot \Vy \}$. Let
$\rho(\M{A})$ denote the spectral radius of a square matrix $\M{A}$, i.e., the maximum of the magnitudes of its eigenvalues.

\subsection{Tensors}
\label{sec:tensors}

A tensor is an $m$-way array.
We let $\RT{m}{n}$ denote the space of $m^\text{th}$-order real-valued tensors
with dimension $n$, e.g., $\RT{3}{2} =
\Real^{2 \times 2 \times 2}$. We adopt the convention that $\RT{0}{n}
= \Real$.

We formally introduce the notion of a symmetric tensor, sometimes also
called supersymmetric, which is invariant under any permutation of its
indices. Further, we define 
a generalization of the tensor-vector multiplication
in equations \Eqn{Axm1} and \Eqn{NLP}.

\begin{definition}[Symmetric tensor \cite{CoGoLiMo08}]
  A tensor $\TA \in \RT{m}{n}$ is \emph{symmetric} if
  \begin{displaymath}
    \TE{a}{i_{p(1)} \cdots i_{p(m)}} = \TE{a}{i_1 \cdots i_m}
    \qtext{for all} i_1, \dots, i_m \in \{1,\dots,n\}
    \qtext{and} p \in \Perm{m}.
  \end{displaymath}
\end{definition}

\begin{definition}[Symmetric tensor-vector multiply]\label{def:mult}
  Let $\TA \in \RT{m}{n}$ be symmetric and $\Vx \in \Rn$.
  Then for $0 \leq r \leq m - 1$, the \emph{$(m-r)$-times product} of the tensor
  $\TA$ with the vector $\Vx$ is denoted by $\TA \Vx^{m-r} \in
  \RT{r}{n}$ and defined by
  \begin{displaymath}
    (\TA \Vx^{m-r})_{i_1 \cdots i_r} \equiv \sum_{i_{r+1}, \dots, i_m} 
    \TE{A}{i_1 \cdots i_m} \VE{x}{i_{r+1}} \cdots \VE{x}{i_m}
    \qtext{for all} i_1, \dots, i_r \in \{1,\dots,n\}.
  \end{displaymath}
\end{definition}

\begin{example}
  The identity matrix plays an important role in matrix analysis. This
  notion can be extended in a sense to the domain of tensors. We may
  define an identity tensor as a symmetric tensor $\T{E} \in \RT{m}{n}$
  such that 
  \begin{displaymath}
    \T{E}\Vx^{m-1} = \Vx \qtext{for all} \Vx \in \US.
  \end{displaymath}
  We restrict $\Vx \in \US$ since it is not possible to have a tensor with $m > 2$
  such that the above equation 
  holds for all $\Vx \in \Rn$.  For any $\Vx \notin \US$, the
  above equation implies
  \begin{displaymath}
    \T{E}\Vx^{m-1} 
    = \| \Vx \|^{m-1} \T{E}(\Vx/\|\Vx\|)^{m-1} 
    = \| \Vx \|^{m-1} (\Vx/\|\Vx\|)
    = \| \Vx \|^{m-2} \Vx.
  \end{displaymath}
  Consider the case of $m=4$ and $n=2$. The system of equations that
  must be satisfied for all $\Vx \in \US$ is
  \begin{align*}
    \TE{e}{1111} \VE{x}{1}^3 + 3\TE{e}{1112} \VE{x}{1}^2 \VE{x}{2}
    + 3\TE{e}{1122} \VE{x}{1} \VE{x}{2}^2 + \TE{e}{1222} \VE{x}{2}^3 
    &= \VE{x}{1}, \\
    \TE{e}{1112} \VE{x}{1}^3 + 3\TE{e}{1122} \VE{x}{1}^2 \VE{x}{2}
    + 3\TE{e}{1222} \VE{x}{1} \VE{x}{2}^2 + \TE{e}{2222} \VE{x}{2}^3 
    &= \VE{x}{2}. 
  \end{align*}
  Consider $\Vx = \begin{bmatrix} 1 & 0 \end{bmatrix}\Tra$. This
  yields $\TE{e}{1111}=1$ and $\TE{e}{1112} = 0$. Similarly, $\Vx =
  \begin{bmatrix} 0 & 1 \end{bmatrix}\Tra$ yields $\TE{e}{2222}=1$ and
  $\TE{e}{1222}=0$. The only remaining unknown is $\TE{e}{1122}$, and
  choosing, e.g., $\Vx = \begin{bmatrix} \sqrt{2}/2 & \sqrt{2}/2
  \end{bmatrix}\Tra$ yields $\TE{e}{1122} = 1/3$. In summary, the
  identity tensor for $m=4$ and $n=2$ is
  \begin{displaymath}
    \TE{e}{ijkl} =
    \begin{cases}
      1 & \text{if } i = j = k = l, \\
      1/3 & \text{if } i = j \ne k = l, \\
      1/3 & \text{if } i = k \ne j = l, \\
      1/3 & \text{if } i = l \ne j = k, \\
      0 & \text{otherwise}.
    \end{cases}
  \end{displaymath}
  We generalize this idea in the next property.
\end{example}

\begin{property}
  For $m$ even, the identity tensor $\T{E} \in \RT{m}{n}$ satisfying
  $\T{E}\Vx^{m-1} = \Vx$ for all $\Vx \in \US$ is given by
  \begin{equation}\label{eq:identity}
    \TE{e}{i_1 \cdots i_m} = 
    \frac{1}{m!}
    \sum_{p \in \Perm{m}} 
    \delta_{i_{p(1)} i_{p(2)}} 
    \delta_{i_{p(3)} i_{p(4)}} 
    \cdots
    \delta_{i_{p(m-1)} i_{p(m)}}
  \end{equation}
  for  $i_1,\dots,i_m \in \{1,\dots,n\}$,
  where $\delta$ is the standard Kronecker delta, i.e.,
  \begin{displaymath}
    \delta_{ij} \equiv
    \begin{cases}
      1 & \text{if } i = j, \\
      0 & \text{if } i \neq j.
    \end{cases}
  \end{displaymath}
\end{property}%

This identity tensor appears in a previous work \cite{Qi05}, where it is denoted by $I_E$ and used
to define a generalization of the characteristic polynomial for symmetric even-order tensors.
  
\begin{example}
  There is no identity tensor for $m$ odd.
  This is seen because if $\T{E}\Vx^{m-1} = \Vx$ for some odd $m$ and some $\Vx \in \US$, then
  we would have $-\Vx \in \US$ but $\T{E}(-\Vx)^{m-1} = \Vx \ne -\Vx$.
\end{example}


For any even-order tensor (i.e., $m$ even), observe that if $(\lambda,\Vx)$ is an eigenpair, then $(\lambda,-\Vx)$ is also an eigenpair since
\begin{displaymath}
  \TA(-\Vx)^{m-1} = -\TA\Vx^{m-1} = \lambda(-\Vx).
\end{displaymath}
Likewise, for any odd-order tensor (i.e., $m$ odd), $(-\lambda,-\Vx)$ is also an eigenpair since
\begin{displaymath}
    \TA(-\Vx)^{m-1} = \TA\Vx^{m-1} = (-\lambda)(-\Vx).
\end{displaymath}
These are \emph{not} considered to be distinct eigenpairs.

We later present, as \Thm{neigs}, a recently derived result \cite{CaSt10} that bounds the number of real eigenpairs by $((m-1)^n-1)/(m-2)$. We defer discussion of this result until \Sec{complex}, where we discuss complex eigenpairs.

Because the tensor eigenvalue equation for $m > 2$ amounts to a system of nonlinear
equations in the components of $\Vx$, a direct solution is
challenging. 
%
Numerical algorithms exist for finding all solutions of a system
of polynomial equations, but become computationally expensive for systems with many
variables (here, large $n$) and with high-order polynomials (here, large $m$).
A polynomial system solver (\texttt{NSolve}) using a Gr\"obner basis method is
available in Mathematica \cite{Wo08} and has been employed to generate a complete list
of eigenpairs for some of the examples in this paper.
The solver is instructed to find
all solutions $(\lambda, \Vx)$ of the system \Eqn{EVP}.
Redundant solutions with the opposite sign of $\Vx$ (for even $m$) or the opposite
signs of $\Vx$ and $\lambda$ (for odd $m$) are then eliminated. 

\subsection{Convex functions}
Convexity theory plays an important role in our analysis. 
Here we recall two important properties of convex functions
\cite{BoVa04}.

\begin{property}[Gradient of convex function]
  \label{prop:convex_gradient}
  A differentiable function $f : \Omega \subseteq \Rn \rightarrow \Real$ is
  convex if and only if $\Omega$ is a convex set and
  $f(\Vy) \geq f(\Vx) + \nabla f(\Vx)\Tra (\Vy-\Vx)$
  for all $\Vx,\Vy \in \Omega$.
\end{property}%

\begin{property}[Hessian of convex function]
  \label{prop:convex_hessian}
  A twice differentiable function $f : \Omega \subseteq \Rn \rightarrow \Real$ is
  convex if and only if $\Omega$ is a convex set and the Hessian\footnote
  {By $\nabla^2$ we denote the Hessian matrix and not its trace, the Laplacian.}
  of $f$ is positive semidefinite on $\Omega$, i.e., 
  $\nabla^2 f(\Vx) \succeq 0$
  for all $\Vx \in \Omega$.
\end{property}

We prove an interesting fact about convex functions on vectors of
unit norm that will prove useful in our later analysis. This fact is
implicit in a proof given previously \cite[Theorem 4]{KoRe02} and
explicit in \cite[Theorem 1]{ReKo03}.
\begin{theorem}[Kofidis and Regalia \cite{KoRe02,ReKo03}]
  \label{thm:cvx}
  Let $f$ be a function that is convex and continuously differentiable
  on $\UB$.  %
  Let $\V{w} \in \US$ with $\nabla f(\V{w}) \neq \V{0}$. If
  $\V{v} = \nabla f(\V{w}) / \|\nabla f(\V{w}) \| $ $\ne \V{w}$, then
  \begin{inlinemath}
    f(\V{v}) - f(\V{w}) > 0.
  \end{inlinemath}
\end{theorem}
\begin{proof}
  For arbitrary nonzero $\V{z} \in \Rn$, $\V{z}\Tra\V{x}$ is strictly maximized for
  $\Vx \in \US$ by $\V{x} = \V{z} / \| \V{z}\|$. 
  Substituting $\V{z} = \nabla f(\V{w})$, it follows that 
  \begin{inlinemath}
    \nabla f(\V{w})\Tra \V{v} > \nabla f(\V{w})\Tra \V{w}, 
  \end{inlinemath}
  since $\V{v} = \nabla f(\V{w}) / \|\nabla f(\V{w}) \| \ne \V{w}$ and $\V{w}
  \in \US$.
  By the convexity of $f$ on $\UB$ and \Prop{convex_gradient}, we have
  \begin{inlinemath}
    f(\V{v}) \geq f(\V{w}) + \nabla f(\V{w})\Tra(\V{v} - \V{w})
  \end{inlinemath}
  for all $\V{v},\V{w} \in \UB$. Consequently,
  \begin{inlinemath}
    f(\V{v}) - f(\V{w}) \geq \nabla f(\V{w})\Tra(\V{v} - \V{w}) > 0.
  \end{inlinemath}
\end{proof}

\subsection{Constrained optimization}

Here we extract relevant theory from constrained optimization
\cite{NoWr99}.

\begin{theorem}
  \label{thm:nlp}
  Let $f:\Rn \rightarrow \Real$ be continuously differentiable.
  A point $\Vx_* \in \US$ is a (constrained) stationary point of
  \begin{displaymath}
    \max f(\Vx) \qtext{subject to} \Vx \in \US
  \end{displaymath}
  if there exists $\mu_* \in \Real$ such that
  \begin{inlinemath}
    \nabla f(\Vx_*) + \mu_* \Vx_* = \V{0}.
  \end{inlinemath}
  The point $\Vx_*$ is a (constrained) isolated local maximum if,
  additionally,  
  \begin{displaymath}
    \V{w}\Tra (\nabla^2 f(\Vx_*) + \mu_* \M{I})\V{w} < 0
    \qtext{for all} \V{w} \in \US \cap \Vx_*^{\bot}.
  \end{displaymath}
\end{theorem}%
\begin{proof}
  The constraint $\Vx \in \US$ can be expressed as $c(\Vx) =
  \frac{1}{2} (\Vx\Tra\Vx - 1) = 0$. The Lagrangian for the
  constrained problem is then given by
  \begin{displaymath}
    \mathcal{L}(\Vx,\mu) = f(\Vx) + \mu c(\Vx).
  \end{displaymath}
  Its first and second derivatives with respect to $\Vx$ are
  \begin{displaymath}
    \nabla \mathcal{L}(\Vx,\mu) = \nabla f(\Vx) + \mu\Vx
    \qtext{and}
    \nabla^2 \mathcal{L}(\Vx,\mu) = \nabla^2 f(\Vx) + \mu \M{I}.
  \end{displaymath}
  By assumption, $\nabla \mathcal{L}(\Vx_*,\mu_*) = \V{0}$ and
  $c(\Vx_*) = 0$. Therefore, the pair $(\Vx_*,\mu_*)$
  satisfies the Karush-Kuhn-Tucker (KKT) conditions \cite[Theorem
  12.1]{NoWr99} and so is a constrained stationary point.  It is
  additionally a constrained isolated local maximum if it meets the second-%
  order sufficient condition \cite[Theorem 12.6]{NoWr99}.
\end{proof}

\subsection{Fixed point theory}

We consider the properties of iterations of the form
\begin{displaymath}
  \Vx_{k+1} = \phi(\Vx_k).
\end{displaymath}
Under certain conditions, the iterates are guaranteed to converge
to a fixed point. In particular, we are interested in ``attracting''
fixed points. 

\begin{definition}[Fixed point]
  A point $\Vx_* \in \Rn$ is a \emph{fixed point} of $\phi: \Rn
  \rightarrow \Rn$ if $\phi(\Vx_*) = \Vx_*$. 
  Further, $\Vx_*$ is an \emph{attracting} fixed point if there exists
  $\delta > 0$ such that the sequence $\{\Vx_k\}$ defined by
  $\Vx_{k+1} = \phi(\Vx_k)$ converges to $\Vx_*$ for any $\Vx_0$
  such that $\|\Vx_0 - \Vx_*\| \leq \delta$.
\end{definition}

\begin{theorem}[{\cite[Theorem 2.8]{Rh74}}]
  \label{thm:fixed_point}
  Let $\Vx_* \in \Rn$ be a fixed point of $\phi: \Rn
  \rightarrow \Rn$, and let $J:\Rn \rightarrow \Real^{n \times
    n}$ be the Jacobian of $\phi$. Then $\Vx_*$ is an attracting
  fixed point if $\sigma \equiv \rho(J(\Vx_*)) < 1$; further, if
  $\sigma > 0$, then the
  convergence of $\Vx_{k+1} = \phi(\Vx_k)$ to $\Vx_*$ is linear with rate $\sigma$.
\end{theorem}

This condition on the Jacobian for an attracting fixed point is sufficient but not necessary. In
particular, if $\sigma \equiv \rho(J(\Vx_*)) = 1$, then $\Vx_*$ may or may not be
attracting, but there is no neighborhood of \emph{linear} convergence to it.
For $\sigma < 1$, the rate of linear convergence depends on $\sigma$ and is
slower for $\sigma$ values closer to 1.
On the other hand, for $\sigma > 1$, an attractor is ruled out by the following.

\begin{theorem}[{\cite[Theorem 1.3.7]{StHu98}}]
  \label{thm:unstable_fixed_point}
  Let $\Vx_* \in \Rn$ be a fixed point of $\phi: \Rn
  \rightarrow \Rn$, and let $J:\Rn \rightarrow \Real^{n \times
    n}$ be the Jacobian of $\phi$. Then $\Vx_*$ is an unstable fixed point
  if $\sigma \equiv \rho(J(\Vx_*)) > 1$.
\end{theorem}

\section{Symmetric higher-order power method (S-HOPM)}
\label{sec:shopm}

We review the symmetric higher-order power method (S-HOPM), introduced
by De~Lathauwer et al.\@~\cite{DeDeVa00a} and analyzed further by
Kofidis and Regalia \cite{KoRe02}.
The purpose of S-HOPM is to solve the optimization problem
\begin{equation}
  \max_{\Vx \in \Rn} |\TA\Vx^m| \qtext{subject to} 
  \Vx \in \US.
\end{equation}
The solution of this problem will be a solution of either the following
maximization problem (lacking the absolute value) or its opposite
minimization problem:
\begin{equation}
  \label{eq:f}
  \max_{\Vx \in \Rn} f(\Vx) \qtext{subject to} 
  \Vx \in \US, \qtext{where} f(\Vx) = \TA\Vx^m.
\end{equation}
Setting $\lambda = f(\Vx)$, these problems are equivalent to finding the best
symmetric rank-1 approximation of a symmetric tensor $\TA \in
\RT{m}{n}$, i.e., 
\begin{equation}
  \label{eq:rank_one}
  \min_{\lambda, \Vx}
  \| \TA - \T{B} \|
  \qtext{subject to} 
  \TE{b}{i_1 \dots i_m} = \lambda \VE{x}{i_1} \cdots \VE{x}{i_m}
  \qtext{and}
  \Vx \in \US.
\end{equation}
Details of the connection
between \Eqn{f} and \Eqn{rank_one} are available elsewhere \cite{DeDeVa00a}.
The S-HOPM algorithm is shown in \Alg{shopm}. We discuss its connection to
the eigenvalue problem in \Sec{f} and its convergence properties in
\Sec{shopm_analysis}. 

\begin{algorithm}
  \caption{Symmetric higher-order power method (S-HOPM) \cite{DeDeVa00a,KoRe02}}
  \label{alg:shopm}
    Given a symmetric tensor $\TA \in \RT{m}{n}$.
  \begin{algorithmic}[1]
    \Require $\Vx_0 \in \Rn$ with $\| \Vx_0 \| = 1$. Let
    $\lambda_0 = \TA \Vx_0^{m}$.
    \For{$k=0,1,\dots$}
    \State $\hat \Vx_{k+1} \gets \TA \Vx_k^{m-1}$
    \State $\Vx_{k+1} \gets \hat \Vx_{k+1} / \| \hat \Vx_{k+1} \|$
    \State $\lambda_{k+1} \gets \TA \Vx_{k+1}^{m}$
    \EndFor
  \end{algorithmic}
\end{algorithm}

\subsection{Properties of $f(\Vx) = \TA\Vx^m$}
\label{sec:f}

The function $f(\Vx) = \TA\Vx^m$ plays an important role in the
analysis of eigenpairs of $\TA$ because all eigenpairs are constrained
stationary points of $f$, as we show below.

We first need to derive the gradient of $f$.
This result is perhaps generally well known \cite[Equation
4]{Li05}, but here we 
provide a proof.

\begin{lemma}\label{lem:g}
  Let $\TA \in \RT{m}{n}$ be symmetric. The gradient of
  $f(\Vx) = \TA\Vx^m$ is
  \begin{equation}
    \label{eq:g}
    g(\Vx) \equiv \nabla f(\Vx) = m \, \TA \Vx^{m-1} \in \Rn.
  \end{equation}
\end{lemma}%
\begin{proof}
We use the basic relation
\begin{inlinemath}
\nabla_k \VE{x}{j} = \delta_{jk}.
\end{inlinemath}
Applying the product rule to \Eqn{f}, we find
\begin{displaymath}
\nabla_k f(\Vx) = \sum_{i_1,\dots,i_m} \sum_{q=1}^m \TE{A}{i_1 i_2\cdots i_m} \VE{x}{i_1} \VE{x}{i_2}\cdots \VE{x}{i_{q-1}}\delta_{i_q k} \VE{x}{i_{q+1}}\cdots \VE{x}{i_m}.
\end{displaymath}
Upon bringing the sum over $q$ to the outside, we observe that for each $q$ the dummy indices $i_1$ and $i_q$ can be interchanged (without affecting the symmetric tensor $\TA$), and the result is independent of $q$:
\begin{displaymath}
\begin{split}
\nabla_k f(\Vx) &= \sum_{q=1}^m \sum_{i_1,\dots,i_m} \TE{A}{i_1 i_2\cdots i_m} \delta_{i_1 k} \VE{x}{i_2}\cdots \VE{x}{i_{q-1}} \VE{x}{i_q} \VE{x}{i_{q+1}}\cdots \VE{x}{i_m}\\
&= \sum_{q=1}^m \sum_{i_2,\dots,i_m} \TE{A}{k i_2\cdots i_m} \VE{x}{i_2}\cdots \VE{x}{i_m}\\
&= m (\TA \Vx^{m-1})_k.
\end{split}
\end{displaymath}
Hence, $\nabla f(\Vx) = m \, \TA \Vx^{m-1}$.
\end{proof}

\begin{theorem}\label{thm:equiv}
  Let $\TA \in \RT{m}{n}$ be symmetric. Then $(\lambda,\Vx)$ is an
  eigenpair of $\TA$ if and only if $\Vx$ is a constrained stationary
  point of \Eqn{f}.
\end{theorem}%
\begin{proof}
  By \Thm{nlp}, 
  any constrained stationary point $\Vx_*$ of \Eqn{f} must satisfy
  $m \, \TA \Vx_*^{m-1} + \mu_* \Vx_* = 0$
  for some $\mu_* \in \Real$. Thus, $\lambda_* = -\mu_*/m$ is the
  eigenvalue corresponding to $\Vx_*$. Conversely, any eigenpair meets
  the condition for being a constrained stationary point with $\mu_* = -m\lambda_*$.
\end{proof}

This is is the connection between
\Eqn{f} and the eigenvalue problem.
It will also be useful to consider the Hessian of $f$, which we
present here.

\begin{lemma}\label{lem:H}
  Let $\TA \in \RT{m}{n}$ be symmetric. The Hessian of
  $f(\Vx) = \TA\Vx^m$ is
  \begin{equation}
    \label{eq:H}
    H(\Vx) \equiv \nabla^2 f(\Vx) = m(m-1) \TA \Vx^{m-2} \in
    \Real^{n \times n}.
  \end{equation}
\end{lemma}%
\begin{proof}
  The $(j,k)$ entry of $H(\Vx)$ is given by the $k^\text{th}$ entry of
  $\nabla g_j(\Vx)$. The function $g_j(\Vx)$ can be rewritten as
  \begin{displaymath}
    g_j(\Vx) = m \sum_{i_2,\dots,i_m}  
    \TE{a}{j i_2 \cdots i_m} 
    \VE{x}{i_2} \cdots \VE{x}{i_m}
    = m \, \T{B}^{(j)} \Vx^{m-1}
  \end{displaymath}
  where $\T{B}^{(j)}$ is the order-$(m-1)$ symmetric tensor that
  is the $j^\text{th}$ subtensor of $\TA$, defined by $\TE{B}{i_1 \cdots
    i_{m-1}}^{(j)} = \TE{A}{j i_1 \cdots i_{m-1}}$. From \Lem{g}, we
  have
  \begin{displaymath}
    \nabla g_j(\Vx) = m(m-1) \T{B}^{(j)} \Vx^{m-2}.
  \end{displaymath}
  Consequently,
  \begin{displaymath}
    (H(\Vx))_{jk} = m(m-1) \sum_{i_3,\dots,i_m} \TE{A}{j k i_3 \cdots
    i_m} \VE{x}{i_3} \cdots \VE{x}{i_m},
  \end{displaymath}
  that is, $H(\Vx) = m(m-1)\TA\Vx^{m-2}$.
\end{proof}

From \Thm{nlp}, we know that the projected Hessian of the Lagrangian
plays a role in determining whether or not a fixed point is a local
maximum or minimum. In our case, since $\mu_* = -m
\lambda_*$, for any eigenpair $(\lambda_*,\Vx_*)$ (which must
correspond to a constrained stationary point by \Thm{equiv})
we have
\begin{displaymath}
  \nabla^2 \mathcal{L}(\Vx_*,\lambda_*) = 
  m(m-1)\TA\Vx_*^{m-2} - m \lambda_* \M{I}.
\end{displaymath}
Specifically, \Thm{nlp} is concerned with the behavior of the Hessian
of the Lagrangian in the subspace orthogonal to $\Vx_*$. 
Thus, we define the projected Hessian of the Lagrangian as
\begin{equation} \label{eq:C}
  C(\lambda_*,\Vx_*) \equiv 
  \M{U}_*\Tra \left((m-1) \TA\Vx_*^{m-2} - \lambda_* \M{I}\right) \M{U}_* \in \Real^{(n-1) \times (n-1)},
\end{equation}
where the columns of $\M{U}_*\in\Real^{n \times (n-1)}$ form an orthonormal basis for
$\Vx_*^{\bot}$. Note that we have removed a factor of $m$ for
convenience. We now classify eigenpairs according to the spectrum of
$C(\lambda_*,\Vx_*)$. 
The import of this classification will be made clear in \Sec{fp}.

\begin{definition}
  Let $\TA \in \RT{m}{n}$ be a symmetric tensor.  We say an eigenpair
  $(\lambda,\Vx)$ of $\TA \in \RT{m}{n}$ is \emph{positive stable} if
  $C(\lambda,\Vx)$ is positive definite, \emph{negative stable} if
  $C(\lambda,\Vx)$ is negative definite, and \emph{unstable} if
  $C(\lambda,\Vx)$ is indefinite.
\end{definition}

These labels are not exhaustive because we do not name the cases where
$C(\lambda,\Vx)$ is only semidefinite, with a zero eigenvalue. Such cases
do not occur for generic tensors.

If $m$ is odd, then $(\lambda,\Vx)$ is positive stable if and only if 
$(-\lambda,-\Vx)$ is negative stable, even though these eigenpairs are
in the same equivalence class. On the other hand, if $m$ is even, then
$(\lambda,\Vx)$ is a positive (negative) stable eigenpair if and only if
$(\lambda,-\Vx)$ is also positive (negative) stable.

\subsection{S-HOPM convergence analysis}
\label{sec:shopm_analysis}

S-HOPM has been deemed unreliable \cite{DeDeVa00a} because convergence
is not guaranteed.
Kofidis and Regalia \cite{KoRe02} provide an analysis explaining that
S-HOPM will converge if certain conditions are met, as well as an example
where the method does not converge, which we reproduce here.

\begin{example2}{{Kofidis and Regalia \cite[Example 1]{KoRe02}}}\label{ex:KoRe02_ex1}
  Let $\TA \in \RT{4}{3}$ be the symmetric tensor defined by 
  \begin{align*}
a_{1111} &=  \phantom{-}0.2883, & 
a_{1112} &= -0.0031, & 
a_{1113} &=  \phantom{-}0.1973, & 
a_{1122} &= -0.2485,\\
a_{1123} &= -0.2939, & 
a_{1133} &=  \phantom{-}0.3847, & 
a_{1222} &=  \phantom{-}0.2972, & 
a_{1223} &=  \phantom{-}0.1862,\\
a_{1233} &=  \phantom{-}0.0919, & 
a_{1333} &= -0.3619, & 
a_{2222} &=  \phantom{-}0.1241, & 
a_{2223} &= -0.3420,\\
a_{2233} &=  \phantom{-}0.2127, & 
a_{2333} &=  \phantom{-}0.2727, & 
a_{3333} &= -0.3054.
  \end{align*}
  Kofidis and Regalia \cite{KoRe02} observed that \Alg{shopm} does not
  converge for this tensor. 
  Because this problem is small, all eigenpairs can be calculated by
  Mathematica as described in \Sec{tensors}.
  From \Thm{neigs}, this problem has at most 13 eigenpairs; we list the 11 real eigenpairs in 
  \Tab{KoRe02_ex1}.
  We ran 100 trials of S-HOPM using different random starting
  points $\Vx_0$ chosen from a uniform
    distribution on $[-1,1]^n$. For these experiments,
  we allow up to 1000 iterations and say that the algorithm has
  converged if $|\lambda_{k+1} - \lambda_k | < 10^{-16}$.  In every
  single trial for this tensor, the algorithm failed to converge. In \Fig{KoRe02_ex1},
  we show an example $\{\lambda_k\}$ sequence with $\Vx_0 =
  \begin{bmatrix} -0.2695 & 0.1972 & 0.3370 \end{bmatrix}\Tra$. This
  coincides with the results reported previously \cite{KoRe02}.
\end{example2}

\begin{table}[htbp]
  \centering
  \caption{Eigenpairs for $\T{A} \in \RT{4}{3}$ from \Ex{KoRe02_ex1}.}
  \label{tab:KoRe02_ex1}
  \footnotesize
  \begin{tabular}{|c|c|c|c|} \hline
$\lambda$ & $\Vx\Tra$ & Eigenvalues of $C(\lambda,\Vx)$ & Type \\ \hline
 $\phantom{-}0.8893$  & [ $\phantom{-}0.6672$  $\phantom{-}0.2471$  $-0.7027$ ] &  $\{$  $-0.8857$,  $-1.8459$  $\}$ & Neg.~stable \\  \hline
 $\phantom{-}0.8169$  & [ $\phantom{-}0.8412$  $-0.2635$  $\phantom{-}0.4722$ ] &  $\{$  $-0.9024$,  $-2.2580$  $\}$ & Neg.~stable \\  \hline
 $\phantom{-}0.5105$  & [ $\phantom{-}0.3598$  $-0.7780$  $\phantom{-}0.5150$ ] &  $\{$  $\phantom{-}0.5940$,  $-2.3398$  $\}$ & Unstable \\  \hline
 $\phantom{-}0.3633$  & [ $\phantom{-}0.2676$  $\phantom{-}0.6447$  $\phantom{-}0.7160$ ] &  $\{$  $-1.1765$,  $-0.5713$  $\}$ & Neg.~stable \\  \hline
 $\phantom{-}0.2682$  & [ $\phantom{-}0.6099$  $\phantom{-}0.4362$  $\phantom{-}0.6616$ ] &  $\{$  $\phantom{-}0.7852$,  $-1.1793$  $\}$ & Unstable\\  \hline
 $\phantom{-}0.2628$  & [ $\phantom{-}0.1318$  $-0.4425$  $-0.8870$ ] &  $\{$  $\phantom{-}0.6181$,  $-2.1744$  $\}$ & Unstable \\  \hline
 $\phantom{-}0.2433$  & [ $\phantom{-}0.9895$  $\phantom{-}0.0947$  $-0.1088$ ] &  $\{$  $-1.1942$,  $\phantom{-}1.4627$  $\}$ & Unstable \\  \hline
 $\phantom{-}0.1735$  & [ $\phantom{-}0.3357$  $\phantom{-}0.9073$  $\phantom{-}0.2531$ ] &  $\{$  $-1.0966$,  $\phantom{-}0.8629$  $\}$ & Unstable \\  \hline
 $-0.0451$  & [ $\phantom{-}0.7797$  $\phantom{-}0.6135$  $\phantom{-}0.1250$ ] &  $\{$  $\phantom{-}0.8209$,  $\phantom{-}1.2456$  $\}$ & Pos.~stable \\  \hline
 $-0.5629$  & [ $\phantom{-}0.1762$  $-0.1796$  $\phantom{-}0.9678$ ] &  $\{$  $\phantom{-}1.6287$,  $\phantom{-}2.3822$  $\}$ & Pos.~stable \\  \hline
 $-1.0954$  & [ $\phantom{-}0.5915$  $-0.7467$  $-0.3043$ ] &  $\{$  $\phantom{-}1.8628$,  $\phantom{-}2.7469$  $\}$ & Pos.~stable\\  \hline
  \end{tabular}
\end{table}

\begin{figure}[htbp]
  \centering
  \includegraphics{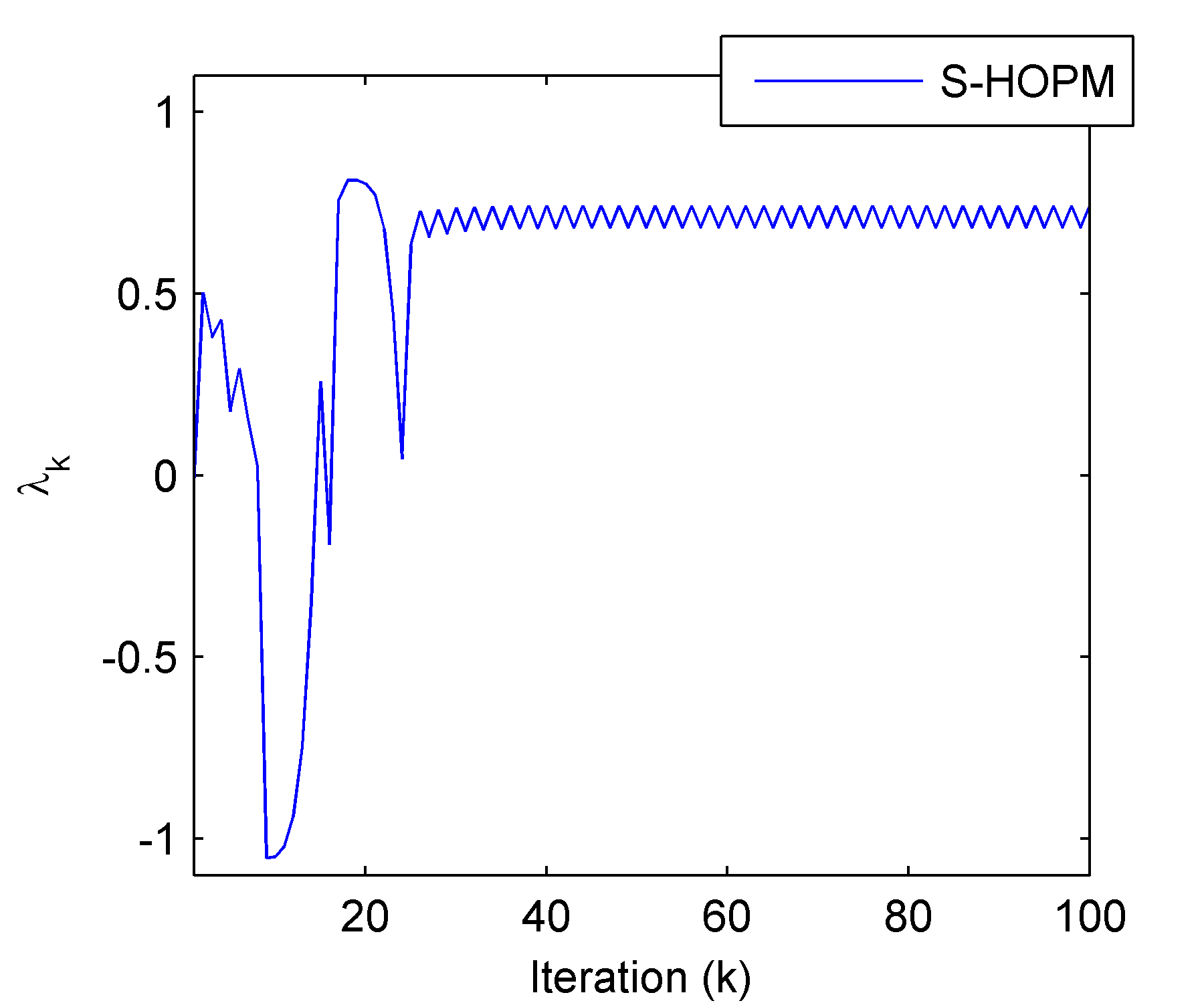}
  \caption{Example $\lambda_k$ values for S-HOPM on $\T{A} \in
  \RT{4}{3}$ from \Ex{KoRe02_ex1}.}
  \label{fig:KoRe02_ex1}
\end{figure}

\begin{example}
  \label{ex:odd}
  As a second illustrative example, we consider an odd-order tensor
  $\TA \in \RT{3}{3}$ defined by
  \begin{align*}
a_{111} &= -0.1281, & 
a_{112} &=  \phantom{-}0.0516, & 
a_{113} &= -0.0954, & 
a_{122} &= -0.1958,\\
a_{123} &= -0.1790, & 
a_{133} &= -0.2676, & 
a_{222} &=  \phantom{-}0.3251, & 
a_{223} &=  \phantom{-}0.2513,\\
a_{233} &=  \phantom{-}0.1773, & 
a_{333} &=  \phantom{-}0.0338.
  \end{align*}
  From \Thm{neigs}, $\TA$ has at most 7 eigenpairs; in this case we
  achieve that bound and the eigenpairs are  
  listed in \Tab{odd}. We ran 100 trials of S-HOPM as
  described for \Ex{KoRe02_ex1}. Every trial converged to either
  $\lambda=0.8730$ or $\lambda=0.4306$, as summarized in
  \Tab{odd-zero}. Therefore, S-HOPM finds 2 of the 7 possible
  eigenvalues. 
\end{example}

\begin{table}[htbp]
  \centering
  \caption{Eigenpairs for $\T{A} \in \RT{3}{3}$ from \Ex{odd}.}
  \label{tab:odd}
  \footnotesize
  \begin{tabular}{|c|c|c|c|}\hline
$\lambda$ & $\Vx\Tra$ & Eigenvalues of $C(\lambda,\Vx)$ & Type \\ \hline
 $0.8730$  & [ $-0.3922$  $\phantom{-}0.7249$  $\phantom{-}0.5664$ ] &  $\{$  $-1.1293$,  $-0.8807$  $\}$ & Neg.~stable \\  \hline
 $0.4306$  & [ $-0.7187$  $-0.1245$  $-0.6840$ ] &  $\{$  $-0.4420$,  $-0.8275$  $\}$ & Neg.~stable \\  \hline
 $0.2294$  & [ $-0.8446$  $\phantom{-}0.4386$  $-0.3070$ ] &  $\{$  $-0.2641$,  $\phantom{-}0.7151$  $\}$ & Unstable \\  \hline
 $0.0180$  & [ $\phantom{-}0.7132$  $\phantom{-}0.5093$  $-0.4817$ ] &  $\{$  $-0.4021$,  $-0.1320$  $\}$ & Neg.~stable \\  \hline
 $0.0033$  & [ $\phantom{-}0.4477$  $\phantom{-}0.7740$  $-0.4478$ ] &  $\{$  $-0.1011$,  $\phantom{-}0.2461$  $\}$ & Unstable \\  \hline
 $0.0018$  & [ $\phantom{-}0.3305$  $\phantom{-}0.6314$  $-0.7015$ ] &  $\{$  $\phantom{-}0.1592$,  $-0.1241$  $\}$ & Unstable \\  \hline
 $0.0006$  & [ $\phantom{-}0.2907$  $\phantom{-}0.7359$  $-0.6115$ ] &  $\{$  $\phantom{-}0.1405$,  $\phantom{-}0.0968$  $\}$ & Pos.~stable\\  \hline
  \end{tabular}
\end{table}

\begin{table}[htbp]
  \centering
  \caption{Eigenpairs for $\TA\in\RT{3}{3}$ from \Ex{odd} computed by S-HOPM with 100
    random starts.}
  \label{tab:odd-zero}
  \footnotesize
  \begin{tabular}{|c|c|c|c|}
    \hline
\# Occurrences & $\lambda$ & $\Vx$ & Median Its. \\ \hline
  62 &  $\phantom{-}0.8730$  & [ $-0.3922$  $\phantom{-}0.7249$  $\phantom{-}0.5664$ ] &   19 \\ \hline 
  38 &  $\phantom{-}0.4306$  & [ $-0.7187$  $-0.1245$  $-0.6840$ ] &  184 \\ \hline 
  \end{tabular}
\end{table}

In their analysis, Kofidis and Regalia \cite{KoRe02} proved that the
sequence $\{\lambda_k\}$ in   
\Alg{shopm} converges if $\TA \in \RT{m}{n}$ is even-order and the function
$f(\Vx)$ is convex or concave on $\Rn$.
Since $m=2\ell$ (because $m$ is even), $f$ can be expressed as
\begin{displaymath}
  f(\Vx) = 
  (\, \underbrace{\Vx \otimes \cdots \otimes \Vx}_{\text{$\ell$ times}} \, )\Tra
  \M{A}\,
  (\, \underbrace{\Vx \otimes \cdots \otimes \Vx}_{\text{$\ell$ times}} \, ),
\end{displaymath}
where $\M{A} \in \Real^{n^{\ell} \times n^{\ell}}$ is an unfolded version
of the tensor $\TA$.\footnote{Specifically, $\M{A} \equiv \M{A}_{(\mathcal{R}
  \times \mathcal{C})}$ with $\mathcal{R} = \{1,\dots,\ell\}$ and
$\mathcal{C} = \{\ell+1,\dots,m\}$ in matricization notation
\cite{Ko06}.} Since $\TA$ is symmetric, it follows that $\M{A}$ is
symmetric. The condition that $f$ is convex (concave) is satisfied
if the Hessian
\begin{displaymath}
  \nabla^2 f(\Vx) =
  (\, \M{I} \otimes \underbrace{\Vx \otimes \cdots \otimes \Vx}_{\text{$\ell-1$ times}} \, )\Tra
  \M{A}\,
  (\, \M{I} \otimes \underbrace{\Vx \otimes \cdots \otimes \Vx}_{\text{$\ell-1$ times}} \, )
\end{displaymath}
is positive (negative) semidefinite for all $\Vx \in \Rn$.

We make a few notes regarding these results. First, even though
$f$ is convex, its restriction to the nonconvex set $\US$ is not.
Second, $\{\lambda_k\}$ is increasing if $f$ is
convex and decreasing if $f$ is concave.
Third, only $\{\lambda_k\}$ is proved to converge for S-HOPM \cite[Theorem 4]{KoRe02};
the iterates $\{\Vx_k\}$ may not. In particular, it is easy to observe
that the sign of $\Vx_k$ may flip back and forth if the concave case
is not handled correctly.

\section{Shifted symmetric higher-order power method (SS-HOPM)}
\label{sec:sshopm}
In this section, we show that S-HOPM can be modified by adding a
``shift'' that guarantees that the method will always converge to an
eigenpair. In the context of ICA, this idea has also been proposed by
Regalia and Kofidis \cite{ReKo03} and Erdogen \cite{Er09}.
Based on the observation that
S-HOPM is guaranteed to converge if the underlying
function is convex or concave on $\Rn$, our method works with a 
suitably modified function
\begin{equation}
  \label{eq:hatf}
  \hat f(\Vx) \equiv f(\Vx) + \alpha (\Vx\Tra\Vx)^{m/2}.
\end{equation}
Maximizing $\hat f$ on $\US$ is the same as maximizing $f$ plus a
constant, yet the properties of the modified function force convexity
or concavity and consequently guarantee
convergence to a KKT point (not necessary the \emph{global} maximum or minimum).
Note that previous papers \cite{ReKo03,Er09} have proposed 
similar shifted functions that are essentially of the form
\begin{inlinemath}
    \hat f(\Vx) \equiv f(\Vx) + \alpha \Vx\Tra\Vx,
\end{inlinemath}
differing only in the exponent.

An advantage of our choice of $\hat f$ in \Eqn{hatf} is that, for even $m$, it
can be interpreted as 
\begin{displaymath}
  \hat f(\Vx) = \hat\TA \Vx^m \equiv (\TA + \alpha \T{E})\Vx^m,
\end{displaymath}
where $\T{E}$ is the identity tensor as defined in \Eqn{identity}.
Thus, for even $m$, our proposed method can be interpreted as S-HOPM applied
to a modified tensor that directly satisfies the convexity properties
to guarantee convergence \cite{KoRe02}. 
Because $\T{E} \Vx^{m-1} = \Vx$ for $\Vx \in \US$, the eigenvectors of
$\hat\TA$ are the same as those of $\TA$ and the eigenvalues are
shifted by $\alpha$. 
Our results, however, are for both odd- and even-order tensors.

\Alg{sshopm} presents the shifted symmetric higher-order power
method (SS-HOPM). Without loss of generality, we assume that a positive
shift ($\alpha \geq 0$) is used to make the modified function in
\Eqn{hatf} convex and a 
negative shift ($\alpha < 0$) to make it concave. 
We have two key results.
\Thm{main} shows that for any starting point $\Vx_0 \in \US$, the
sequence $\{\lambda_k\}$ produced by 
\Alg{sshopm} is guaranteed to converge to an eigenvalue in the
convex case if 
\begin{equation}
  \label{eq:beta}
  \alpha > \beta(\TA) 
  \equiv (m-1) \cdot \max_{\Vx \in \US} \rho(\TA\Vx^{m-2}).
\end{equation}
\Cor{main} handles the concave case where we require $\alpha < -\beta(\TA)$.
\Thm{fp} further shows that \Alg{sshopm} in the convex case will
generically converge to an eigenpair $(\lambda, \Vx)$ that is
negative stable.
\Cor{fp} proves that
\Alg{sshopm} in the concave case will generically converge to an eigenpair
that is positive stable. Generally, neither version will
converge to an eigenpair that is unstable.

\begin{algorithm}
  \caption{Shifted Symmetric  Higher-Order Power Method (SS-HOPM)}
  \label{alg:sshopm}
    Given a tensor $\TA \in \RT{m}{n}$.
  \begin{algorithmic}[1]
    \Require $\Vx_0 \in \Rn$ with $\| \Vx_0 \| = 1$. Let
    $\lambda_0 = \TA \Vx_0^{m}$.
    \Require $\alpha \in \Real$
    \For{$k=0,1,\dots$}
    \If{$\alpha \geq 0$}
    \State $\hat \Vx_{k+1} \gets \TA \Vx_k^{m-1} + \alpha \Vx_k$ %
    \Comment{Assumed Convex}
    \Else
    \State $\hat \Vx_{k+1} \gets -(\TA \Vx_k^{m-1} + \alpha \Vx_k)$
    \Comment{Assumed Concave}
    \EndIf
    \State $\Vx_{k+1} \gets \hat \Vx_{k+1} / \| \hat \Vx_{k+1} \|$
    \State $\lambda_{k+1} \gets \TA \Vx_{k+1}^{m}$
    \EndFor
  \end{algorithmic}
\end{algorithm}

\subsection{SS-HOPM convergence analysis}
We first establish a few key lemmas that guide the choice of the shift
$\alpha > \beta(\TA)$ in SS-HOPM\@.

\begin{lemma}\label{lem:beta_bound}
  Let $\TA \in \RT{m}{n}$ be symmetric and let $\beta(\TA)$ be as
  defined in \Eqn{beta}. Then $\beta(\TA) \leq (m-1)
  \sum_{i_1,\dots,i_m} |\TE{a}{i_1\dots i_m}|$.
\end{lemma}%
\begin{proof}
For all $\Vx,\Vy \in \US$, we obtain
$|\Vy\Tra(\TA\Vx^{m-2})\Vy| \le \sum_{i_1,\dots,i_m} |\TE{a}{i_1\dots i_m}|$
by applying the triangle inequality to the sum of $n^m$ terms. Thus
$\rho(\TA\Vx^{m-2}) \le \sum_{i_1,\dots,i_m} |\TE{a}{i_1\dots i_m}|$
for all $\Vx \in \US$, and the result follows.
\end{proof}

\begin{lemma}\label{lem:f_bound}
  Let $\TA \in \RT{m}{n}$ be symmetric, let $f(\Vx)=\TA\Vx^m$, and let $\beta(\TA)$
  be as defined in \Eqn{beta}.
  Then $|f(\Vx)| \leq \beta(\TA)/(m-1)$ for all $\Vx \in \US$.
\end{lemma}%
\begin{proof}
We have $|\TA\Vx^m| = |\Vx\Tra(\TA\Vx^{m-2})\Vx| \le \rho(\TA\Vx^{m-2}) \le \beta(\TA)/(m-1)$.
\end{proof}

The preceding lemma upper bounds the magnitude of any eigenvalue of $\TA$ by
$\beta(\TA)/(m-1)$ since any eigenpair $(\lambda,\Vx)$ satisfies
$\lambda = f(\Vx)$. Thus, choosing $\alpha > \beta(\TA)$ implies 
that $\alpha$ is greater than the magnitude of any eigenvalue of $\TA$.

\begin{lemma}\label{lem:H_bound}
  Let $\TA \in \RT{m}{n}$ be symmetric and let $H(\Vx)$ and $\beta(\TA)$
  be as defined in \Eqn{H} and \Eqn{beta}.  
  Then $\rho(H(\Vx)) \leq m \beta(\TA)$ for all $\Vx \in \US$.
\end{lemma}%
\begin{proof}
This follows directly from \Eqn{H} and \Eqn{beta}.
\end{proof}

The following theorem proves that \Alg{sshopm} will always converge. Choosing $\alpha >
(m-1) \sum_{i_1,\dots,i_m} |\TE{A}{i_1\dots i_m}|$ is a conservative choice that is
guaranteed to work by \Lem{beta_bound}, but this may slow down
convergence considerably, as we show in subsequent analysis and examples.  

\begin{theorem}\label{thm:main}
  Let $\TA \in \RT{m}{n}$ be symmetric.  For $\alpha > \beta(\TA)$,
  where $\beta(\TA)$ is defined in \Eqn{beta},
  the iterates $\{\lambda_k,\Vx_k\}$ produced by \Alg{sshopm} satisfy the
  following properties. 
  \begin{inparaenum}[(a)]
  \item The sequence $\{\lambda_k\}$ is nondecreasing, and there exists
    $\lambda_*$ such that $\lambda_k \rightarrow \lambda_*$. 
  \item The sequence $\{\Vx_k\}$ has an accumulation point.
  \item For every such accumulation point $\Vx_*$, the pair $(\lambda_*,\Vx_*)$ is an eigenpair of $\TA$.
  \item \label{finite} If $\TA$ has finitely many real eigenvectors, then there exists $\Vx_*$ such that $\Vx_k \rightarrow \Vx_*$.
  \end{inparaenum}
\end{theorem}

\begin{proof}
  Our analysis depends on the modified function $\hat f$ defined in
  \Eqn{hatf}. Its gradient and Hessian for $\Vx \ne \V{0}$ are 
  \begin{align}
    \hat g(\Vx) &\equiv \nabla \hat f(\Vx) = 
    g(\Vx) + m \alpha (\Vx\Tra\Vx)^{m/2-1} \Vx,\\
    \hat H(\Vx) &\equiv \nabla^2 \hat f(\Vx) = 
    H(\Vx) + m \alpha (\Vx\Tra\Vx)^{m/2-1} \M{I} + 
    m(m - 2) \alpha (\Vx\Tra\Vx)^{m/2-2} \Vx\Vx\Tra,
  \end{align}
  where $g$ and $H$ are the gradient and Hessian of $f$ from
  \Lem{g} and \Lem{H}, respectively. And because $\hat f(\Vx) = O(\|\Vx\|^m)$,
  as $\Vx \to \V{0}$, it follows that $\hat f(\Vx)$ is of third or higher order
  in $\Vx$ for $m \ge 3$; thus $\hat g(\V{0}) = \V{0}$ and
  $\hat H(\V{0}) = \M{0}$.
  
  Because it is important for the entire proof, we first show that $\hat f$
  is convex on $\Rn$ for
  $\alpha > \beta(\TA)$. As noted, if $\Vx = \V{0}$, we have $\hat H(\Vx) = \M{0}$ for $m \ge 3$.
  Consider nonzero $\Vx \in \Rn$ and 
  define $\bar\Vx = \Vx / \| \Vx \| \in \US$; then $\hat H(\Vx)$ is positive
  semidefinite (in fact, positive definite) by \Lem{H_bound} since
  \begin{align*}
    \Vy\Tra \hat H(\Vx)\Vy 
    & =  \|\Vx\|^{m-2} \left( \Vy \Tra H(\bar\Vx) \Vy
      + m \alpha + m(m-1) \alpha (\bar\Vx\Tra\Vy)^2 \right) \\
    & \geq \|\Vx\|^{m-2} \left( -m\beta(\TA) + m \alpha + 0 \right) > 0
   \qtext{for all} \Vy \in \US.
  \end{align*}
  By \Prop{convex_hessian}, $\hat f$ is convex on $\Rn$ because its
  Hessian is positive semidefinite.
    
  We also note that $-\alpha$ must be an eigenvalue of $\TA$ if $\hat g(\Vx) =
  \V{0}$ for some $\Vx \in \US$, since 
  \begin{displaymath}
    \hat g(\Vx) = \V{0} 
    \qtext{implies}
    \TA \Vx^{m-1} + \alpha \Vx = \V{0}.
  \end{displaymath}
  By \Lem{f_bound}, choosing $\alpha > \beta(\TA)$ ensures that $\alpha$
  is greater than the magnitude of any eigenvalue, and so $\hat g(\Vx)
  \neq \V{0}$ for all $\Vx \in \US$. This ensures that the update in
  \Alg{sshopm}, which reduces to
  \begin{equation}
    \label{eq:xkp1}
    \Vx_{k+1} = \frac{\hat g(\Vx_k)}{\| \hat g(\Vx_k) \|}
  \end{equation}
  in the convex case, is always well defined.

  \begin{asparaenum}[(a)]
  \item \label{main1} 
  Since $\hat f$ is convex on $\UB$ and $\Vx_{k+1}, \Vx_k
    \in \US$ and $\Vx_{k+1} = \nabla \hat f(\Vx_k) / \| \nabla \hat
    f(\Vx_k) \|$, \Thm{cvx} yields
    \begin{displaymath}
      \lambda_{k+1} - \lambda_k 
      = \hat f(\Vx_{k+1}) - \hat f(\Vx_k)  \geq 0,
    \end{displaymath}
    where the nonstrict inequality covers the possibility that $\Vx_{k+1} = \Vx_k$.
    Thus, $\{\lambda_k\}$ is a nondecreasing sequence. 
    By \Lem{f_bound}, $\lambda_k = f(\Vx_k)$ is bounded, so the sequence
    must converge to a limit point $\lambda_*$.\footnote{Note that
      the similar approach proposed for ICA
      \cite[Theorem 2]{ReKo03} allows the shift $\alpha$ to vary at
      each iteration so long as the underlying function remains convex.}
  \item 
    Since $\{\Vx_k\}$ is an infinite sequence on a compact set $\US$, it
    must have an accumulation point $\Vx_* \in \US$ by the Bolzano-Weierstrass theorem.
    Note also that continuity of $f$ implies that $\lambda_* = \TA\Vx_*^m$.
  \item 
    By part (\ref*{main1}) of the proof, convexity of $\hat f$,
    and \Prop{convex_gradient}, we have 
    \begin{displaymath}
      \hat f(\Vx_{k+1}) - \hat f(\Vx_k) \rightarrow 0
    \end{displaymath}
    and thus
    \begin{displaymath}
      \hat g(\Vx_k)\Tra (\Vx_{k+1} - \Vx_k) 
      \rightarrow 0.
    \end{displaymath}
    Using \Eqn{xkp1}, we can rewrite the above formula as
    \begin{equation}
      \label{eq:gxk}
      \|\hat g(\Vx_k)\| - 
      \hat g(\Vx_k)\Tra \Vx_k 
      \rightarrow 0.
    \end{equation}
    By continuity of $\hat g$, an accumulation point $\Vx_*$ must satisfy
    \begin{equation}
      \label{eq:key}
      \|\hat g(\Vx_*)\| - \hat g(\Vx_*)\Tra \Vx_* = 0,
    \end{equation}
    which implies
    \begin{displaymath}
      \|\hat g(\Vx_*)\| = \hat g(\Vx_*)\Tra \Vx_* = (m\, \TA\Vx_*^{m-1} + m\alpha\Vx_*)\Tra \Vx_* = m(\lambda_* + \alpha).
    \end{displaymath}
    Because $\Vx_* \in \US$, \Eqn{key} can hold only if
    \begin{displaymath}
      \Vx_* = \frac{\hat g(\Vx_*)}{\| \hat g(\Vx_*) \|} = \frac{m\, \TA\Vx_*^{m-1} + m\alpha\Vx_*}{m(\lambda_* + \alpha)},
    \end{displaymath}
    that is,
    \begin{displaymath}
      \TA\Vx_*^{m-1} = \lambda_* \Vx_*.
    \end{displaymath}
    Hence $(\lambda_*, \Vx_*)$ is an eigenpair of $\TA$.
  \item 
    Equation \Eqn{gxk} gives
    \begin{displaymath}
      \|\hat g(\Vx_k)\| (1 - \Vx_{k+1}\Tra \Vx_k) \to 0.
    \end{displaymath}
    Because $\|\hat g(\Vx_k)\|$ is bounded away from 0 and because $\Vx_k,\Vx_{k+1} \in \US$,
    this requires that
    \begin{equation}
      \label{eq:xk0}
      \|\Vx_k - \Vx_{k+1}\| \to 0.
    \end{equation}
    Recall that every accumulation point of $\{\Vx_k\}$ must be a (real)
    eigenvector of $\TA$. If these eigenvectors are
    finite in number and thus isolated, consider removing an arbitrarily small open neighborhood of each
    from $\US$, leaving a closed and thus compact space $Y \subset
    \US$ containing no accumulation points of $\{\Vx_k\}$. If $\{\Vx_k\}$ had infinitely many iterates in $Y$,
    it would have an accumulation point in $Y$ by the Bolzano-Weierstrass theorem, creating a contradiction.
    Therefore at most finitely many iterates are in $Y$, and $\{\Vx_k\}$ is ultimately confined to
    arbitrarily small neighborhoods of the eigenvectors.
    By \Eqn{xk0}, however, $\|\Vx_k - \Vx_{k+1}\|$ eventually remains smaller than the minimum distance between
    any two of these neighborhoods.
    Consequently, the iteration ultimately cannot jump from one neighborhood to another, and so
    in the limit $\{\Vx_k\}$ is confined to an arbitrarily small neighborhood of a \emph{single} eigenvector
    $\Vx_*$, to which it therefore converges.
  \end{asparaenum}  
  Hence, the proof is complete.
\end{proof}

Note that the condition of finitely many real eigenvectors in part (\ref*{finite}) holds for generic tensors.
We conjecture that the convergence of $\{\Vx_k\}$ is guaranteed even without this condition.

\begin{example}
  Again consider $\TA \in \RT{4}{3}$ from \Ex{KoRe02_ex1}.
  We show
  results using a shift of $\alpha = 2$. 
  We ran 100 trials of SS-HOPM using
  the experimental conditions described in \Ex{KoRe02_ex1}. 
  We found 3 real eigenpairs; the results are
  summarized in \Tab{sshopm-convex}. Three example runs (one for each
  eigenvalue) are shown in \Fig{sshopm-convex}. 

  We also considered the ``conservative'' choice of $\alpha = (m-1)
  \sum_{i_1,\dots,i_m} |\TE{A}{i_1\dots i_m}| = 55.6620$. We ran 100
  trials of SS-HOPM using the experimental conditions described in
  \Ex{KoRe02_ex1}, except that we increased the maximum number of
  iterations to 10,000. Every trial converged to one of the same 3
  real eigenpairs, but the number of iterations was around
  1000 (versus around 60 for $\alpha =2$); in \Sec{fp}, we
  see that the rate of convergence asymptotically decreases as
  $\alpha$ increases.

  Analogous results are shown for $\TA \in \RT{3}{3}$ from \Ex{odd}
  with a shift of $\alpha=1$ in \Tab{sshopm-odd-convex} and
  \Fig{sshopm-odd-convex}. Here SS-HOPM finds 2 additional eigenpairs
  compared to S-HOPM\@. In this case, we also considered $\alpha = (m-1)
  \sum_{i_1,\dots,i_m} |\TE{A}{i_1\dots i_m}| = 9.3560$, but this
  again increased the number of iterations up to a factor of ten.
  
  For both tensors,
  $\{\lambda_k\}$ is always a nondecreasing sequence.
  Observe further that SS-HOPM converges only to eigenpairs
  that are negative stable. 
\end{example}

\begin{table}[htbp]
  \centering
  \subfloat[$\TA\in\RT{4}{3}$ from \Ex{KoRe02_ex1} with $\alpha=2$.]{
    \footnotesize
    \begin{tabular}{|c|c|c|c|} \hline
\# Occurrences & $\lambda$ & $\Vx$ & Median Its. \\ \hline
  46 &  $\phantom{-}0.8893$  & [ $\phantom{-}0.6672$  $\phantom{-}0.2471$  $-0.7027$ ] &   63 \\ \hline 
  24 &  $\phantom{-}0.8169$  & [ $\phantom{-}0.8412$  $-0.2635$  $\phantom{-}0.4722$ ] &   52 \\ \hline 
  30 &  $\phantom{-}0.3633$  & [ $\phantom{-}0.2676$  $\phantom{-}0.6447$  $\phantom{-}0.7160$ ] &   65 \\ \hline 
    \end{tabular}
    \label{tab:sshopm-convex}
  }

  \subfloat[$\TA\in\RT{3}{3}$ from \Ex{odd} with $\alpha=1$.]{
    \footnotesize
    \begin{tabular}{|c|c|c|c|} \hline
      \# Occurrences & $\lambda$ & $\Vx$ & Median Its. \\ \hline
      40 &  $\phantom{-}0.8730$  & [ $-0.3922$  $\phantom{-}0.7249$  $\phantom{-}0.5664$ ] &   32 \\ \hline 
      29 &  $\phantom{-}0.4306$  & [ $-0.7187$  $-0.1245$  $-0.6840$ ] &   48 \\ \hline 
      18 &  $\phantom{-}0.0180$  & [ $\phantom{-}0.7132$  $\phantom{-}0.5093$  $-0.4817$ ] &  116 \\ \hline 
      13 &  $-0.0006$  & [ $-0.2907$  $-0.7359$  $\phantom{-}0.6115$ ] &  145 \\ \hline 
    \end{tabular}
    \label{tab:sshopm-odd-convex}
  }
  \caption{Eigenpairs computed by SS-HOPM (convex) with 100 random starts.} 
\end{table}

\begin{figure}[htbp]
  \centering
  \subfloat[$\TA\in\RT{4}{3}$ from \Ex{KoRe02_ex1} with $\alpha=2$.]%
  {\label{fig:sshopm-convex}
    \includegraphics[width=.49\textwidth]{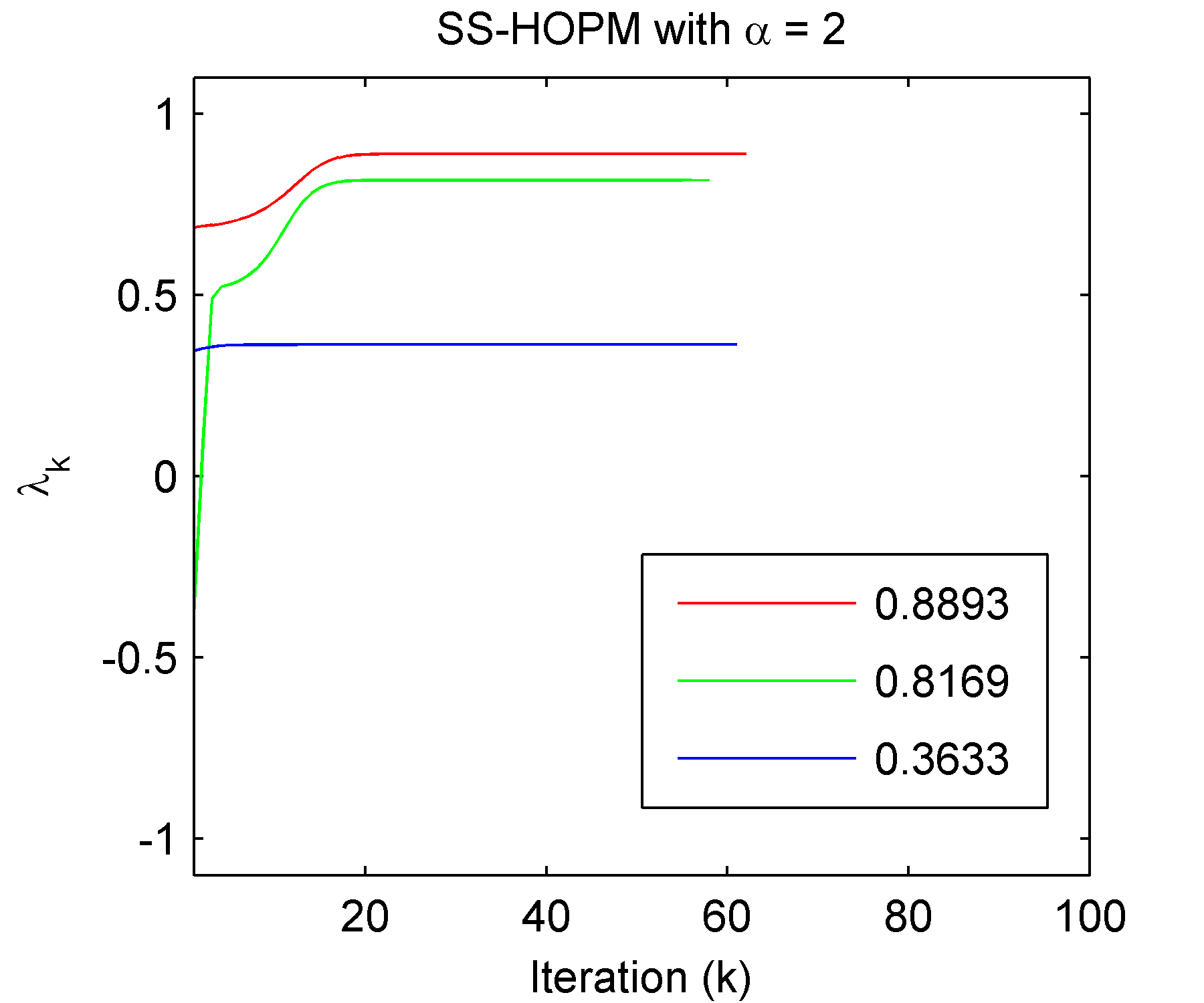}}
  \subfloat[$\TA\in\RT{3}{3}$ from \Ex{odd} with $\alpha=1$.]%
  {\label{fig:sshopm-odd-convex}
    \includegraphics[width=.49\textwidth]{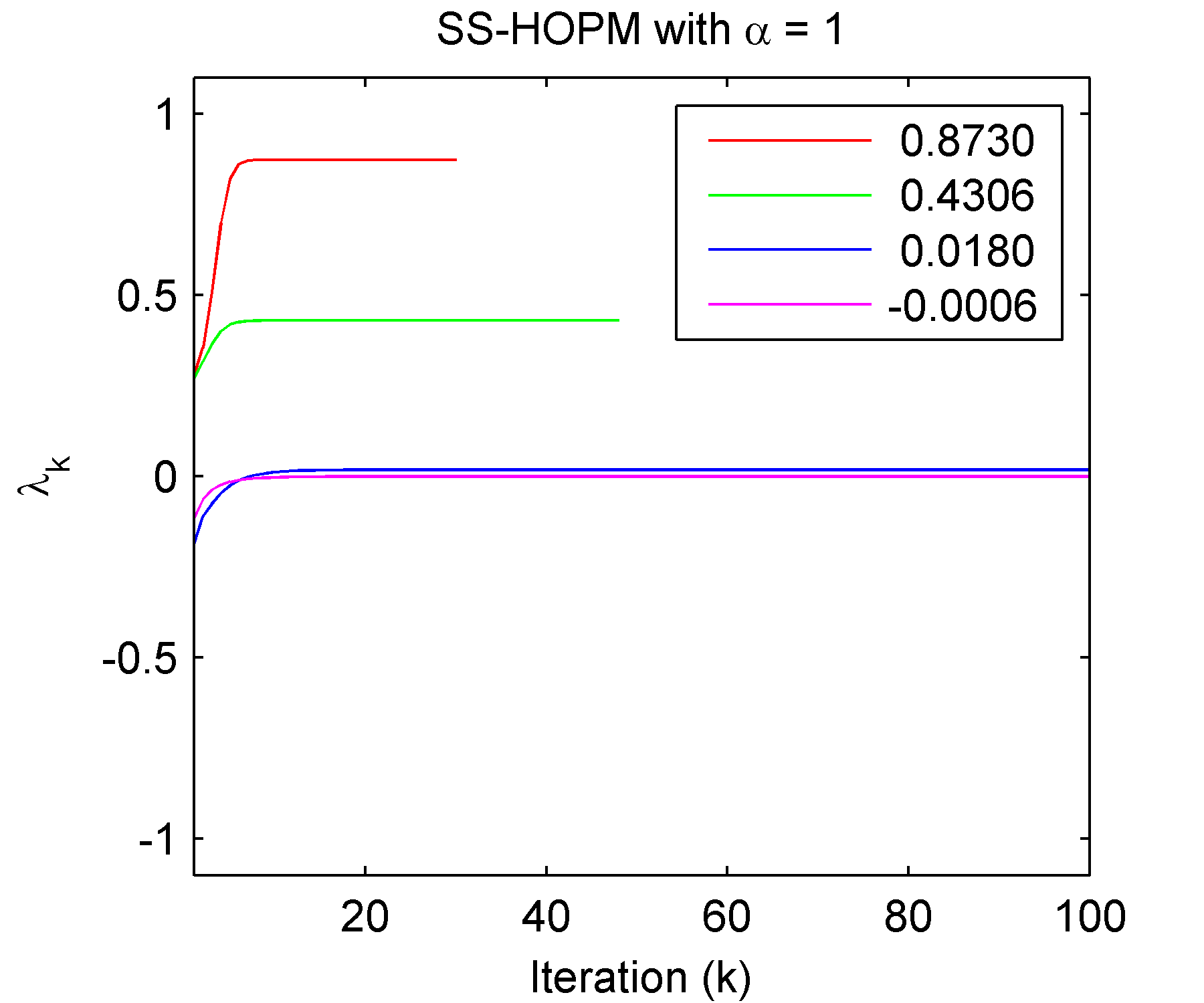}}
  \caption{Example $\lambda_k$ values for SS-HOPM (convex). 
      One sequence is shown for each distinct eigenvalue.}
\end{figure}

Using a large enough negative value of $\alpha$ makes $\hat f$
concave. It was observed \cite{KoRe02} that $f(\Vx) =
f(-\Vx)$ for even-order tensors and so the sequence $\{\lambda_k\}$
converges regardless of correctly handling the minus sign. The only
minor problem in the concave case
is that the sequence of iterates $\{\Vx_k\}$ does not converge.
This is easily fixed, however, by correctly handling the sign as we do in
\Alg{sshopm}.  The corresponding theory for the concave case is
presented in \Cor{main}.  In this case
we choose $\alpha$ to be negative, i.e., the theory suggests $\alpha <
-\beta(\TA)$.

\begin{corollary}
  \label{cor:main}
  Let $\TA \in \RT{m}{n}$ be symmetric.  For $\alpha <
  -\beta(\TA)$, where $\beta(\TA)$ is defined in \Eqn{beta},
  the iterates $\{\lambda_k,\Vx_k\}$ produced by \Alg{sshopm} satisfy the
  following properties. 
  \begin{inparaenum}[(a)]
  \item The sequence $\{\lambda_k\}$ is nonincreasing, and there exists
    $\lambda_*$ such that $\lambda_k \rightarrow \lambda_*$. 
  \item The sequence $\{\Vx_k\}$ has an accumulation point.
  \item For any such accumulation point $\Vx_*$, the pair $(\lambda_*,\Vx_*)$ is an eigenpair of $\TA$.
  \item If the eigenvalues of $\TA$ are isolated, then $\Vx_k \to \Vx_*$.
  \end{inparaenum}
\end{corollary}
\begin{proof}
  Apply the proof of \Thm{main} with $f(\Vx) = - \TA\Vx^m$. 
\end{proof}

\begin{example}
  Revisiting $\TA\in\RT{4}{3}$ in \Ex{KoRe02_ex1} again, we run another 100 trials using
  $\alpha=-2$.  We find 3 (new) real eigenpairs; the results are
  summarized in \Tab{sshopm-concave}.  Three
  example runs (one for each eigenvalue) are shown in
  \Fig{sshopm-concave}.

  We also revisit $\TA \in \RT{3}{3}$ from \Ex{odd} and use
  $\alpha=-1$. In this case, we find the opposites, i.e.,
  $(-\lambda,-\Vx)$, of the eigenpairs found with $\alpha = 1$, as
  shown in \Tab{sshopm-odd-concave}. This is to be expected for
  odd-order tensors since there is symmetry, i.e., $f(\Vx) =
  -f(-\Vx)$, $C(\lambda,\Vx) = -C(-\lambda,-\Vx)$, etc.  Observe that
  the median number of iterations is nearly unchanged; this is
  explained in the subsequent subsection where we discuss the rate of convergence. Four example
  runs (one per eigenvalue) are shown in \Fig{sshopm-odd-concave}. 
  
  The sequence $\{\lambda_k\}$ is
  nonincreasing in every case.  Each of the eigenpairs found in the
  concave case is positive stable.
\end{example}

\begin{table}[htbp]
  \centering
  \subfloat[$\TA\in\RT{4}{3}$ from \Ex{KoRe02_ex1} with $\alpha=-2$.]{
    \footnotesize
    \begin{tabular}{|c|c|c|c|} \hline
\# Occurrences & $\lambda$ & $\Vx$ & Median Its. \\ \hline
  15 &  $-0.0451$  & [ $-0.7797$  $-0.6135$  $-0.1250$ ] &   35 \\ \hline 
  40 &  $-0.5629$  & [ $-0.1762$  $\phantom{-}0.1796$  $-0.9678$ ] &   23 \\ \hline 
  45 &  $-1.0954$  & [ $-0.5915$  $\phantom{-}0.7467$  $\phantom{-}0.3043$ ] &   23 \\ \hline 
    \end{tabular}
    \label{tab:sshopm-concave}
  }

  \subfloat[$\TA\in\RT{3}{3}$ from \Ex{odd} with $\alpha=-1$.]{
    \footnotesize
    \begin{tabular}{|c|c|c|c|} \hline
\# Occurrences & $\lambda$ & $\Vx$ & Median Its. \\ \hline
  19 &  $\phantom{-}0.0006$  & [ $\phantom{-}0.2907$  $\phantom{-}0.7359$  $-0.6115$ ] &  146 \\ \hline 
  18 &  $-0.0180$  & [ $-0.7132$  $-0.5093$  $\phantom{-}0.4817$ ] &  117 \\ \hline 
  29 &  $-0.4306$  & [ $\phantom{-}0.7187$  $\phantom{-}0.1245$  $\phantom{-}0.6840$ ] &   49 \\ \hline 
  34 &  $-0.8730$  & [ $\phantom{-}0.3922$  $-0.7249$  $-0.5664$ ] &   33 \\ \hline 
    \end{tabular}
    \label{tab:sshopm-odd-concave}
  }
  \caption{Eigenpairs computed by SS-HOPM (concave) with 100 random starts.} 
\end{table}

\begin{figure}[htbp]
  \centering
  \subfloat[$\TA\in\RT{4}{3}$ from \Ex{KoRe02_ex1} with $\alpha=2$.]%
  {\label{fig:sshopm-concave}
    \includegraphics[width=.49\textwidth]{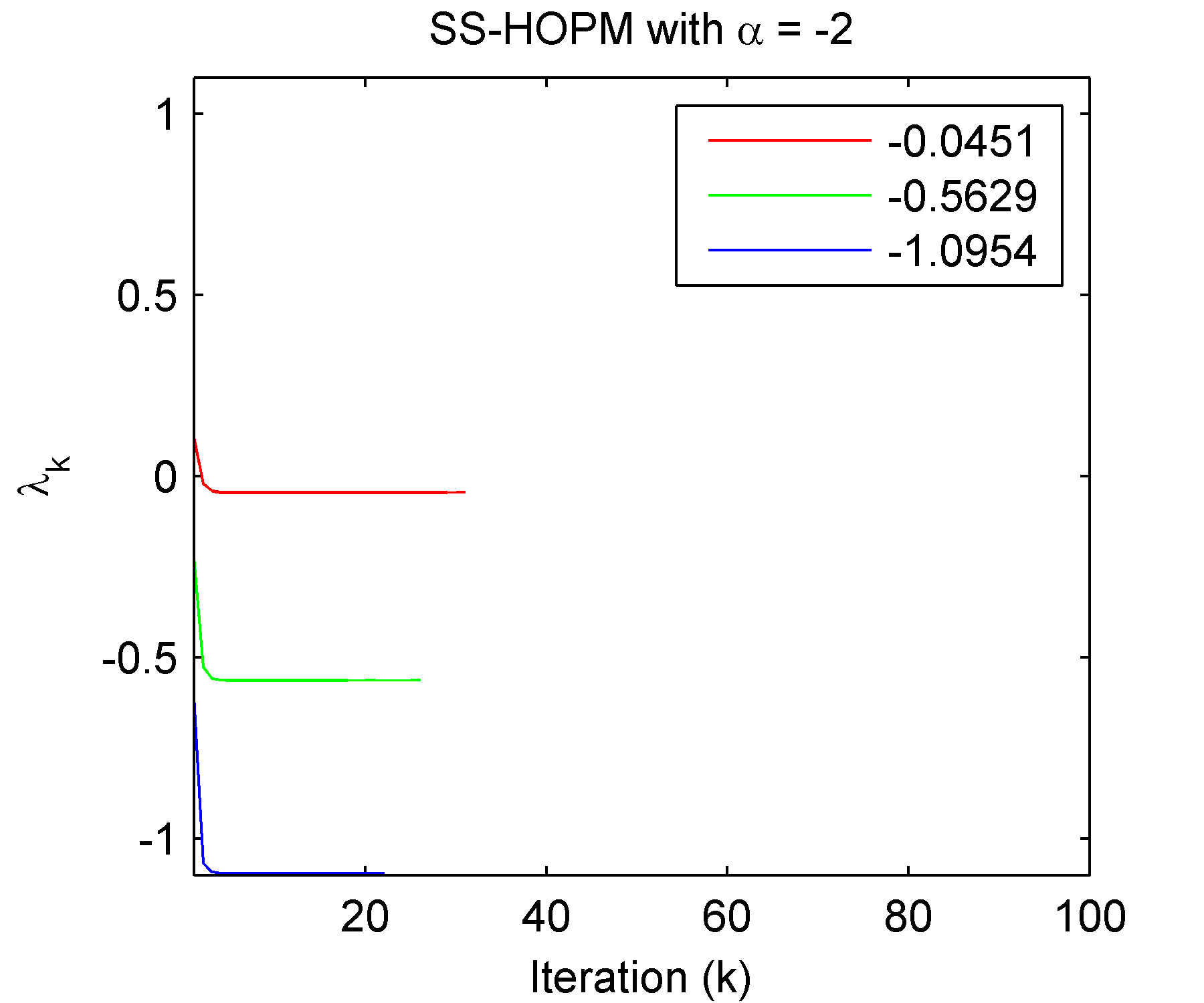}}
  \subfloat[$\TA\in\RT{3}{3}$ from \Ex{odd} with $\alpha=1$.]%
  {\label{fig:sshopm-odd-concave}
    \includegraphics[width=.49\textwidth]{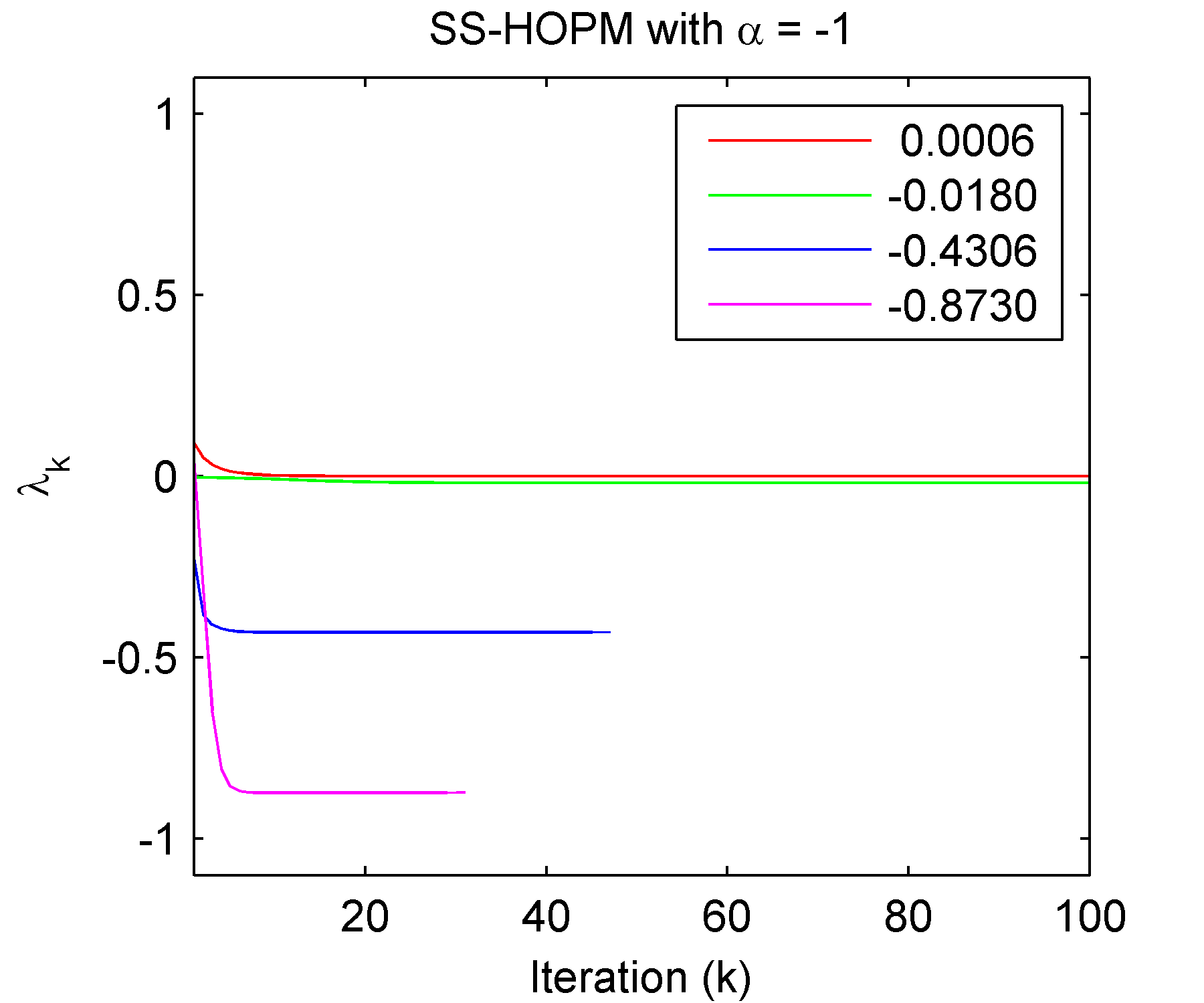}}
  \caption{Example $\lambda_k$ values for SS-HOPM (concave). 
      One sequence is shown for each distinct eigenvalue.}
\end{figure}

\subsection{SS-HOPM fixed point analysis}
\label{sec:fp}
In this section, we show that fixed point analysis allows us to easily
characterize convergence to eigenpairs according to whether they are positive stable,
negative stable, or unstable.
The convex version of SS-HOPM will generically converge to
eigenpairs that are negative stable;
the concave version will generically converge
to eigenpairs that are positive stable.

To justify these conclusions, we consider \Alg{sshopm} in the convex case
as a fixed point iteration $\Vx_{k+1} =
\phi(\Vx_k;\alpha)$, where $\phi$ is defined as
\begin{equation}\label{eq:phi}
  \phi(\Vx;\alpha) = \phi_1(\phi_2(\Vx;\alpha)) 
  \text{ with }
  \phi_1(\Vx) = \frac{\Vx}{(\Vx\Tra\Vx)^{\frac{1}{2}}} 
  \text{ and }
  \phi_2(\Vx;\alpha) = \TA \Vx^{m-1} + \alpha \Vx.
\end{equation}
Note that an eigenpair $(\lambda, \Vx)$ is a fixed point
if and only if $\lambda + \alpha > 0$, which is always true for $\alpha > \beta(\TA)$.

From \cite{Fa05}, the Jacobian of the operator $\phi$ is
\begin{displaymath}
  J(\Vx;\alpha) = \phi_1'(\phi_2(\Vx;\alpha)) \phi_2'(\Vx;\alpha),
\end{displaymath}
where derivatives are taken with respect to $\Vx$ and
\begin{displaymath}
  \phi_1'(\Vx) = \frac{(\Vx\Tra\Vx) \M{I} - \Vx
    \Vx\Tra}{(\Vx\Tra\Vx)^{\frac{3}{2}}}
  \qtext{and}
  \phi_2'(\Vx;\alpha) = (m-1)\TA\Vx^{m-2} + \alpha \M{I}.
\end{displaymath}
At any eigenpair $(\lambda, \Vx)$, we have
\begin{gather*}
  \phi_2(\Vx;\alpha) = (\lambda + \alpha)\Vx
  , \quad
  \phi_1'(\phi_2(\Vx;\alpha)) = \frac{(\M{I} - \Vx\Vx\Tra)}{\lambda + \alpha}
  , \\ 
  \qtext{and}
  \phi_2'(\Vx;\alpha) = (m-1) \TA \Vx^{m-2} + \alpha \M{I}.
\end{gather*}
Thus, the Jacobian at $\Vx$ is 
\begin{equation}
  \label{eq:J}
  J(\Vx;\alpha) = \frac{(m-1)(\TA\Vx^{m-2} - \lambda \Vx\Vx\Tra) +
    \alpha(\M{I} - \Vx\Vx\Tra)}{\lambda + \alpha}.
\end{equation}
Observe that the Jacobian is symmetric.

\begin{theorem}
  \label{thm:fp}
  Let $(\lambda,\Vx)$ be an eigenpair of a symmetric tensor $\TA \in \RT{m}{n}$. 
  Assume $\alpha \in \Real$ such that $\alpha > \beta(\TA)$, where
  $\beta(\TA)$ is as defined in \Eqn{beta}.
  Let $\phi(\Vx)$ be given by \Eqn{phi}.
  Then $(\lambda,\Vx)$ is negative stable if and only if
  $\Vx$ is a linearly attracting fixed point of $\phi$.
\end{theorem}
\begin{proof}
  Assume that $(\lambda, \Vx)$ is negative stable.
  The Jacobian $J(\Vx;\alpha)$ is given by \Eqn{J}.
  By \Thm{fixed_point},
  we need to show that  $\rho(J(\Vx;\alpha)) < 1$ or, equivalently since
  $J(\Vx;\alpha)$ is symmetric,
  $| \Vy\Tra J(\Vx;\alpha) \Vy | < 1$ for all $\Vy \in \US$. We restrict
  our attention to $\Vy \bot \Vx$ since 
  $J(\Vx;\alpha) \Vx = \V{0}$. 

  Let $\Vy \in \US$ with $\Vy \bot \Vx$. Then
  \begin{displaymath}
    |\Vy\Tra J(\Vx;\alpha) \Vy| = 
    \left| 
    \frac%
    { \Vy\Tra \left( (m-1) \TA\Vx^{m-2} \right) \Vy + \alpha  }%
    {\lambda + \alpha}
    \right|
  \end{displaymath}
  The assumption that $(\lambda,\Vx)$ is negative stable means that
  $C(\lambda,\Vx)$ is negative definite; therefore,
  \begin{inlinemath}
    \Vy\Tra \left((m-1)\TA\Vx^{m-2}\right) \Vy < \lambda.
  \end{inlinemath}
  On the other hand, by the definition of $\beta$, 
  \begin{displaymath}
    \rho\!\left((m-1)\TA\Vx^{m-2}\right) \leq \beta(\TA).
  \end{displaymath}
  Thus, using the fact that $\lambda + \alpha$ is positive, we have
  \begin{displaymath}
    0 <
    \frac{-\beta(\TA) + \alpha}{\lambda + \alpha}
    \leq \frac
    { \Vy\Tra \left( (m-1) \TA\Vx^{m-2} \right) \Vy + \alpha  }
    {\lambda + \alpha}
    <
    \frac{\lambda + \alpha}{\lambda + \alpha} = 1
  \end{displaymath}
  Hence, $\rho(J(\Vx;\alpha)) < 1$, and $\Vx$ is a linearly attracting fixed point.

  On the other hand, if $(\lambda,\Vx)$ is not negative stable, then there exists
  $\V{w} \in \US$ such that $\V{w} \bot \Vx$ and $\V{w}\Tra \left( (m-1) \TA\Vx^{m-2} \right)
  \V{w} \geq \lambda$. Thus,
  \begin{displaymath}
    \V{w}\Tra J(\Vx;\alpha) \V{w} = 
    \frac
    { \V{w}\Tra \left( (m-1) \TA\Vx^{m-2} \right) \V{w} + \alpha  }
    {\lambda + \alpha}
    \geq
    \frac{\lambda + \alpha}{\lambda + \alpha} = 1.
  \end{displaymath}
  Consequently, $\rho(J(\Vx;\alpha)) \geq 1$, and $\Vx$ is not a linearly attracting fixed point
  by \Thm{fixed_point} and \Thm{unstable_fixed_point}.
\end{proof}

In fact, we can see from the proof of \Thm{fp} that if the eigenpair $(\lambda,\Vx)$
is not negative stable, there is no choice of
$\alpha \in \Real$ that will make $\rho(J(\Vx;\alpha))<1$. For $\Vx$
to be a fixed point at all, we must have $\lambda + \alpha > 0$, and this is
sufficient to obtain $\rho(J(\Vx;\alpha)) \ge 1$ if $(\lambda,\Vx)$ is not negative stable.
In other words, smaller values of $\alpha$ do not induce ``accidental'' convergence
to any additional eigenpairs.

An alternative argument establishes, for $\alpha > \beta(\TA)$,
the slightly broader result that any
attracting fixed point, regardless of order of convergence, must be a strict constrained
local maximum of $f(\Vx) = \TA \Vx^m$ on $\US$. That is, the marginally attracting case corresponds
to a stationary point that has degenerate $C(\lambda,\Vx)$ but is still a maximum.
This follows from \Thm{cvx}, where the needed convexity holds for $\alpha > \beta(\TA)$, so that any
vector $\Vx' \in \US$ in the neighborhood of convergence of $\Vx$ must satisfy $f(\Vx') < f(\Vx)$.
One can convince oneself that the converse also holds for $\alpha > \beta(\TA)$, i.e., any strict
local maximum corresponds to an attracting fixed point. This is because the strict monotonicity of $f$
under iteration (other than at a fixed point) implies that the iteration acts as a contraction on
the region of closed contours of $f$ around the maximum.

The counterpart of \Thm{fp} for the concave case is as follows.

\begin{corollary}
  \label{cor:fp}
  Let $(\lambda,\Vx)$ be an eigenpair of a symmetric tensor $\TA \in \RT{m}{n}$. 
  Assume $\alpha \in \Real$ such that $\alpha < -\beta(\TA)$, where
  $\beta(\TA)$ is as defined in \Eqn{beta}.
  Let $\phi(\Vx)$ be given by \Eqn{phi}.
  Then $(\lambda,\Vx)$ is positive stable if and only if
  $\Vx$ is a linearly attracting fixed point of $-\phi$.
\end{corollary}

\begin{example}
  We return again to $\TA \in \RT{4}{3}$ as defined in \Ex{KoRe02_ex1}.
  \Fig{jac_KoRe02_ex1} shows the spectral radius of the Jacobian of
  the fixed point iteration for varying values of $\alpha$ for all
  eigenpairs that are positive or negative stable.
  At $\alpha=0$, the spectral radius is greater
  than 1 for every eigenvalue, and this is why S-HOPM never
  converges. At $\alpha=2$, on the other hand, we see that the
  spectral radius is less than 1 for all of the negative stable
  eigenpairs. Furthermore, the spectral radius stays less than 1 as
  $\alpha$ increases. Conversely, at $\alpha=-2$, the spectral radius
  is less than 1 for all the eigenpairs that are positive stable.

  In \Fig{rate_KoRe02_ex1}, we plot example iteration sequences for $\|\Vx_{k+1}-\Vx_*\|/\|\Vx_k - \Vx_*\|$ for each eigenpair, using $\alpha = 2$ for the negative stable eigenpairs and $\alpha = -2$ for the positive stable eigenpairs. 
We expect $\|\Vx_{k+1}-\Vx_*\| = \sigma \|\Vx_k - \Vx_*\|$ where
$\sigma$ is the spectral radius of the Jacobian $J(\Vx;\alpha)$.
For example, for $\lambda = -1.0954$, we have $\sigma=0.4$ (shown as a dashed line) and this precisely matched the observed rate of convergence (shown as a solid line).

  \Fig{sphere_f} plots  $f(\Vx)$ on the unit sphere using color to
  indicate function value. 
  We show the front and back of the sphere. Notice that the
  horizontal axis is from $1$ to $-1$ in the left plot and from $-1$ to $1$ in the
  right plot, as if walking around the sphere. In this image, the
  horizontal axis corresponds to $x_2$ and the vertical axis to $x_3$;
  the left image is centered at $x_1 = 1$ and the right image at
  $x_1=-1$. Since $m$ is even, the function is symmetric, i.e.,
  $f(\Vx) = f(-\Vx)$.
  The eigenvectors are shown as white, gray, and black circles
  corresponding to their classification as negative stable, positive
  stable, and unstable, respectively;
  in turn, these correspond to maxima, minima, and saddle
  points of  $f(\Vx)$. 

  \Fig{sphere_convex} shows the basins of attraction for
  SS-HOPM with $\alpha = 2$. 
  Every grid point on the sphere was used as a starting point for
  SS-HOPM, and it is colored\footnote{Specifically, each block on the
  sphere is colored according to the convergence of its lower left
  point.} according to which eigenvalue it converged
  to. 
  In this case, every run converges to a
  negative stable eigenpair (labeled
  with a white circle). Recall that SS-HOPM
  must converge to some eigenpair per \Thm{main}, and \Thm{fp} says
  that it is generically a negative stable eigenpair. 
  Thus, the
  non-attracting points lie on the boundaries of the domains of
  attraction.

  \Fig{sphere_concave} shows the basins of attraction for SS-HOPM with
  $\alpha = -2$. In this case, every starting point converges to an
  eigenpair that is positive stable (shown
  as gray circles).
\end{example}

\begin{figure}[htbp]
  \centering
  \subfloat[$\TA\in\RT{4}{3}$ from \Ex{KoRe02_ex1}.]{
    \label{fig:jac_KoRe02_ex1}
    \includegraphics[width=2.4in,trim=0 0 0 0,clip]{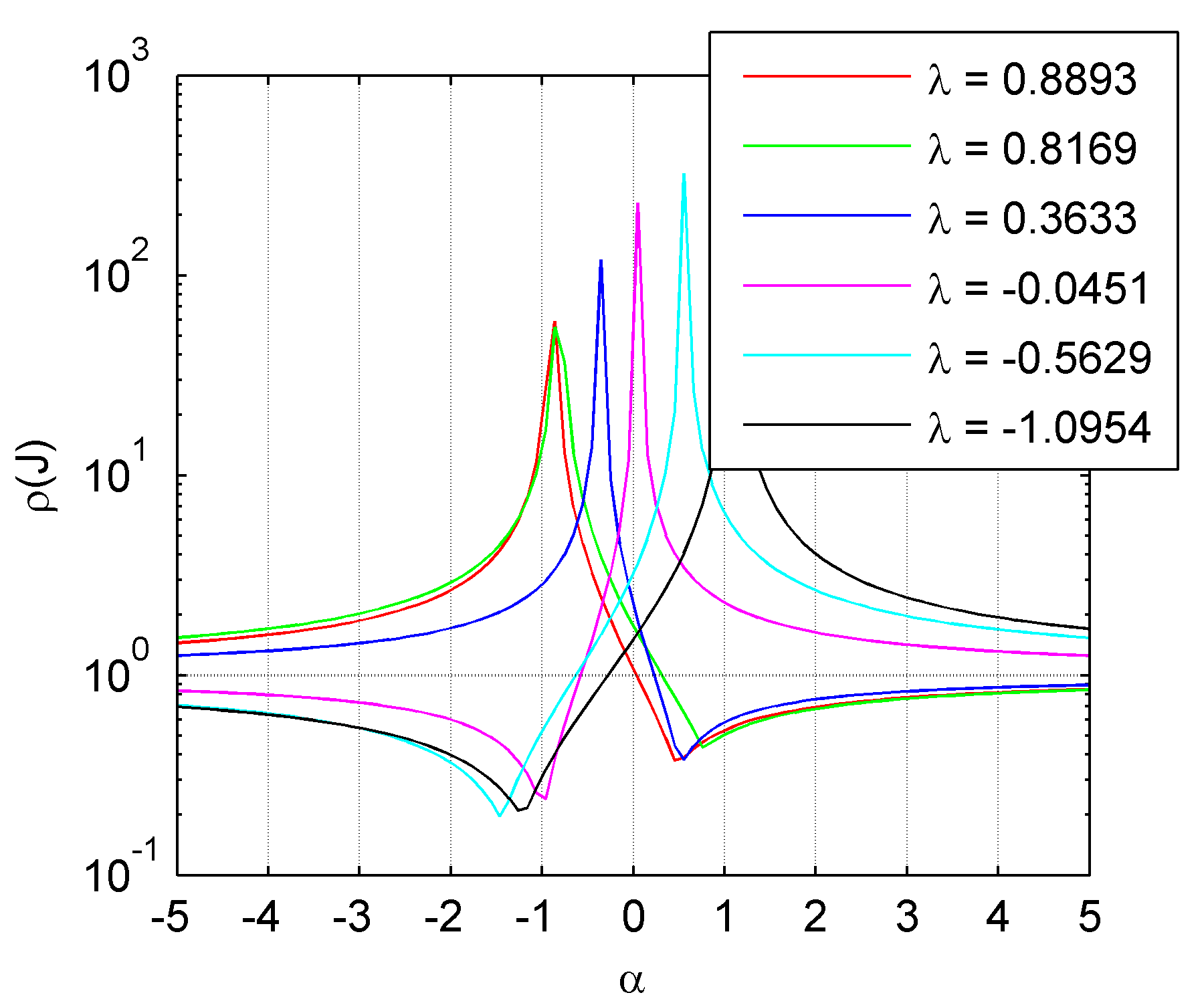}
  }
  \subfloat[$\TA\in\RT{3}{3}$ from \Ex{odd}.]{
    \label{fig:jac_odd}
    \includegraphics[width=2.4in,trim=0 0 0 0,clip]{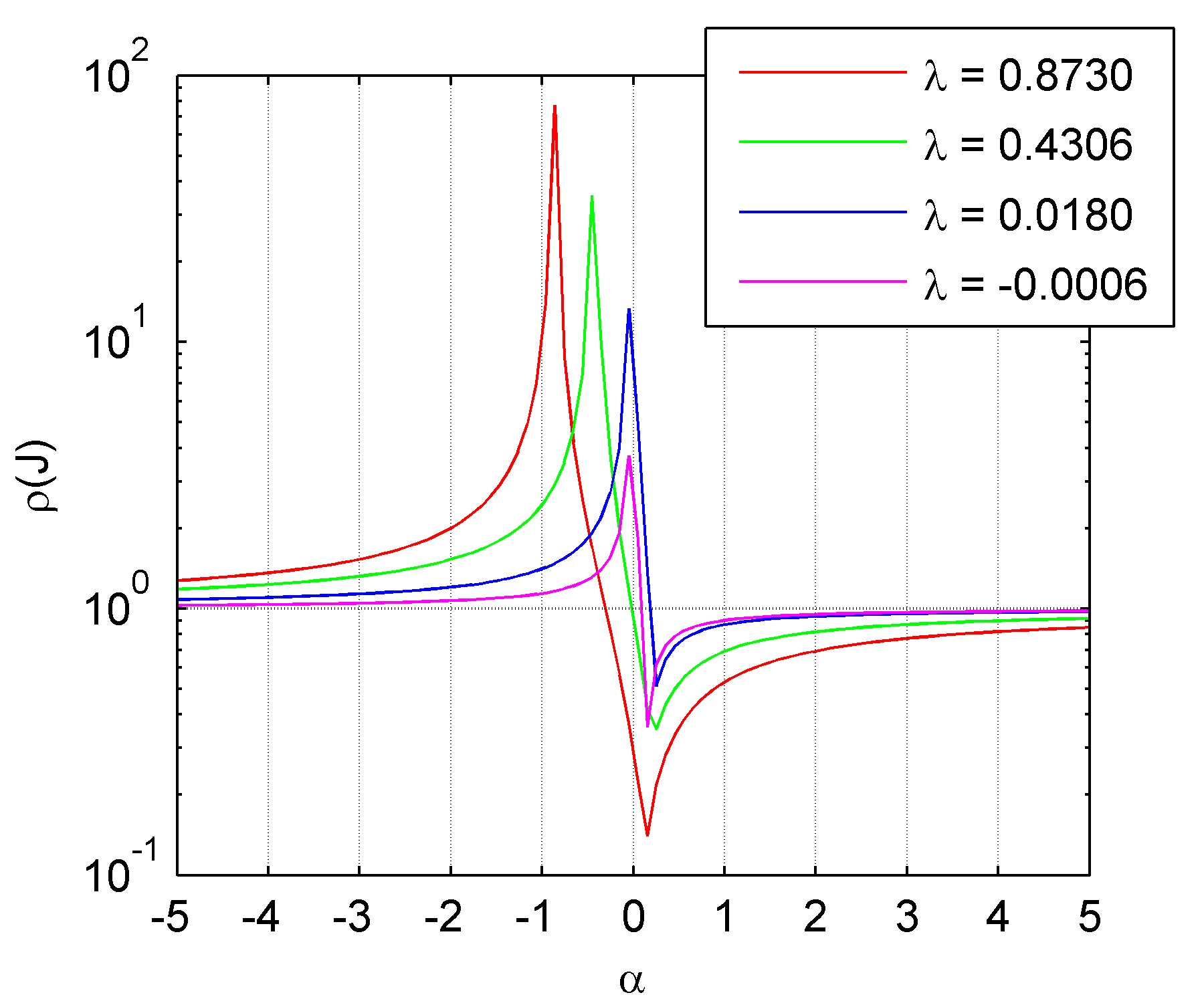}
  }
  \caption{Spectral radii of the Jacobian $J(\Vx;\alpha)$ for
  different eigenpairs as $\alpha$ varies.}
\end{figure}

\begin{figure}[htbp]
  \centering
  \subfloat[$\TA\in\RT{4}{3}$ from \Ex{KoRe02_ex1} using $\alpha = \pm 2$, as appropriate.]{
    \label{fig:rate_KoRe02_ex1}
    \includegraphics[width=2.4in,trim=0 0 0 5,clip]{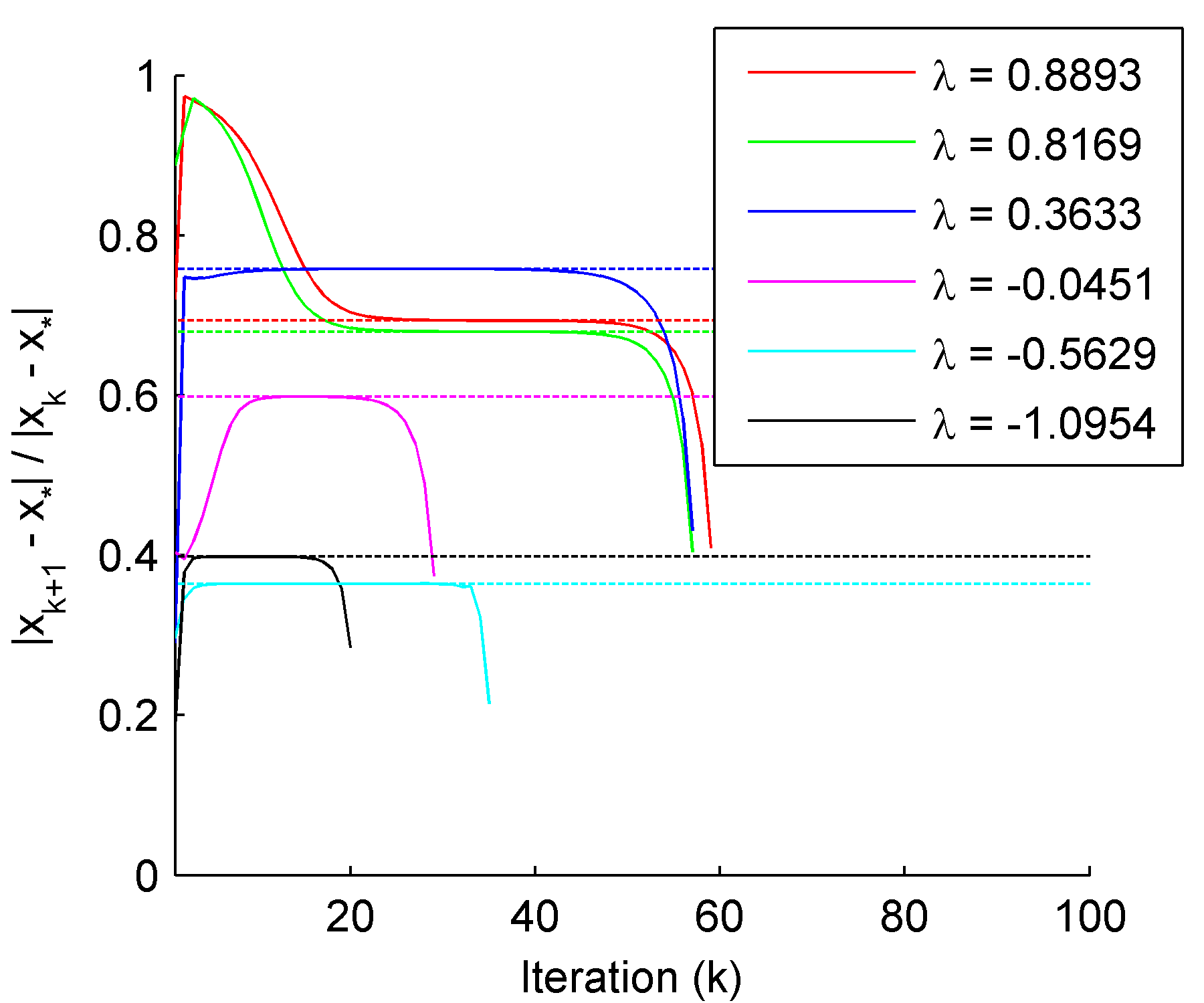}
  }
  \subfloat[$\TA\in\RT{3}{3}$ from \Ex{odd} using $\alpha = 1$.]{
    \label{fig:rate_odd}
    \includegraphics[width=2.4in,trim=0 0 0 5,clip]{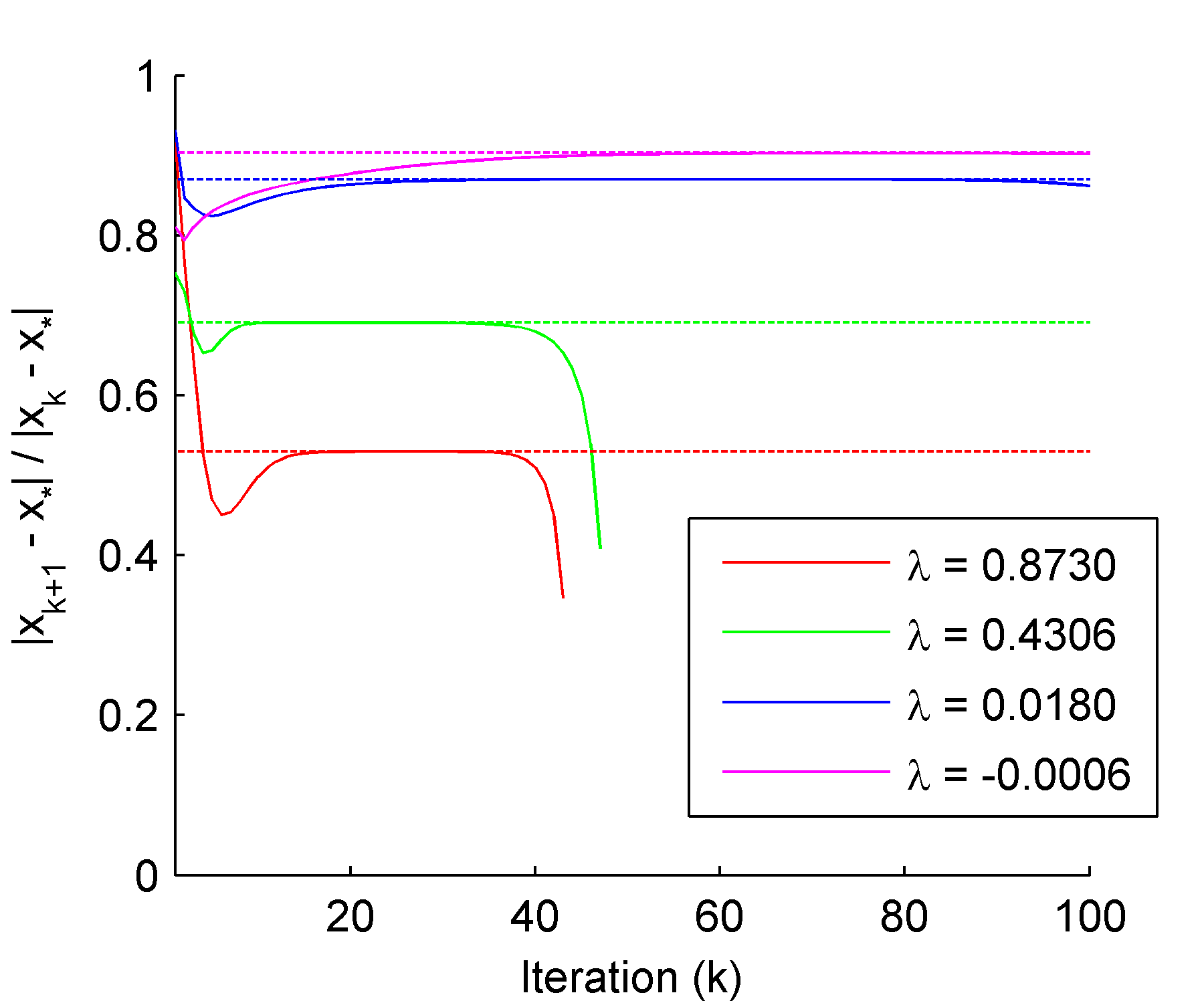}
  }
  \caption{Example plots of $\|\Vx_{k+1}-\Vx_*\|/\|\Vx_k - \Vx_*\|$. 
    The expected rate of convergence from $J(\Vx_*;\alpha)$ is shown
    as a dashed line.} 
  \label{fig:rate}
\end{figure}

\begin{figure}[htbp]
  \centering
  \subfloat[Function values for $f(\Vx)=\TA\Vx^m$.]{
    \label{fig:sphere_f}
    \includegraphics[trim=35 0 35 25,clip,height=0.27\textheight]{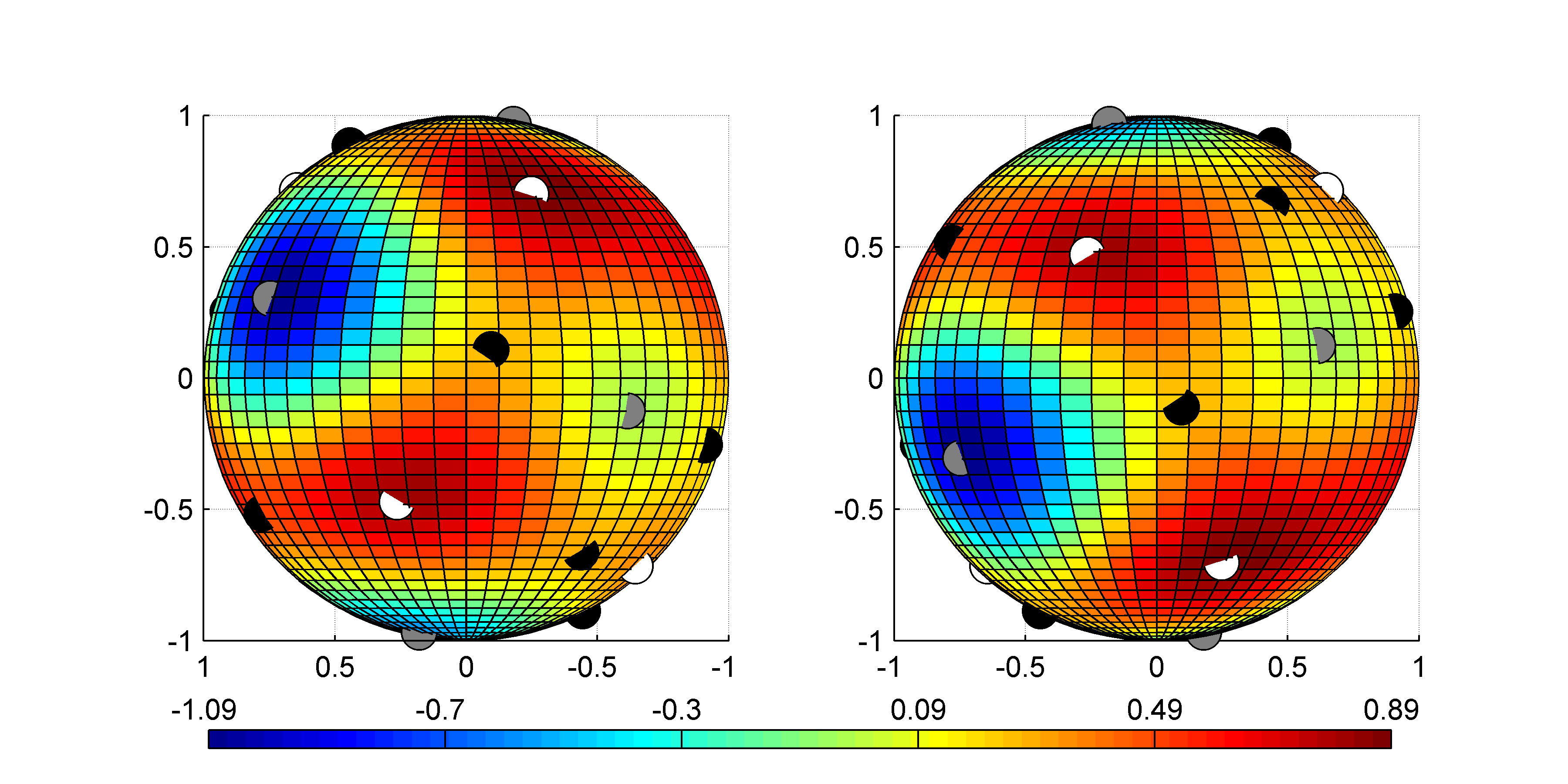}
  }

  \subfloat[SS-HOPM basins of attraction using $\alpha = 2$.]
  {
    \label{fig:sphere_convex}
    \includegraphics[trim=35 0 35 25,clip,height=0.27\textheight]{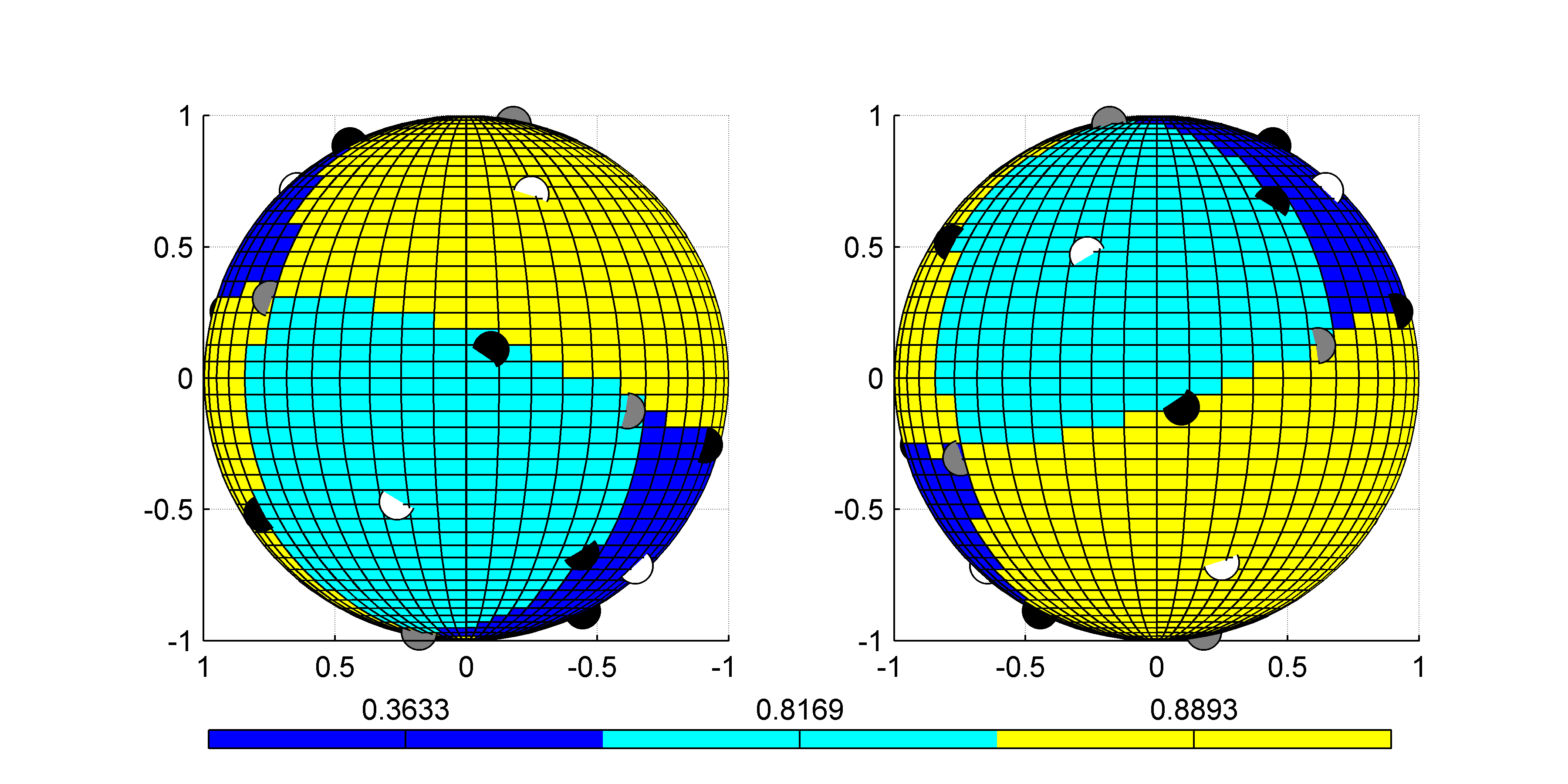}
  }

  \subfloat[SS-HOPM basins of attraction using $\alpha = -2$.]
  {
    \label{fig:sphere_concave}
    \includegraphics[trim=35 0 35 25,clip,height=0.27\textheight]{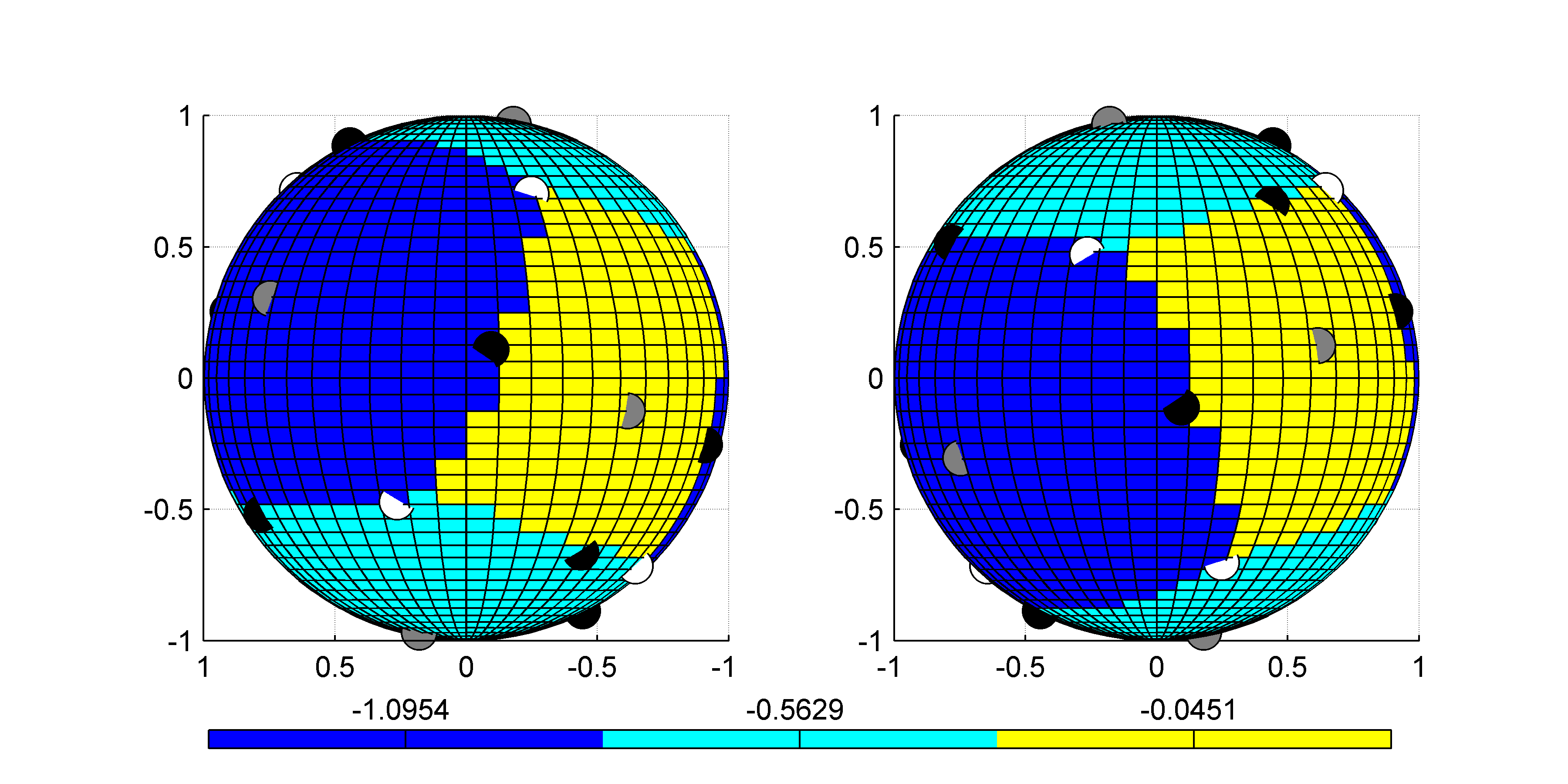}
  }

  \caption{Illustrations for $\TA \in \RT{4}{3}$ from
    \Ex{KoRe02_ex1}. The horizontal axis corresponds 
    to $\VE{x}{2}$ and the vertical axis to $\VE{x}{3}$; the left
    image is centered at $\VE{x}{1}=1$ and the right at
    $\VE{x}{1}=-1$.  White, gray, and black dots indicate eigenvectors
    that are negative stable, positive stable, and unstable,
    respectively.}
\end{figure}

\begin{figure}[htbp]
  \centering
  \subfloat[Function values for $f(\Vx)=\TA\Vx^m$.]{
    \label{fig:sphere_odd_f}
    \includegraphics[trim=35 0 35 25,clip,height=0.27\textheight]{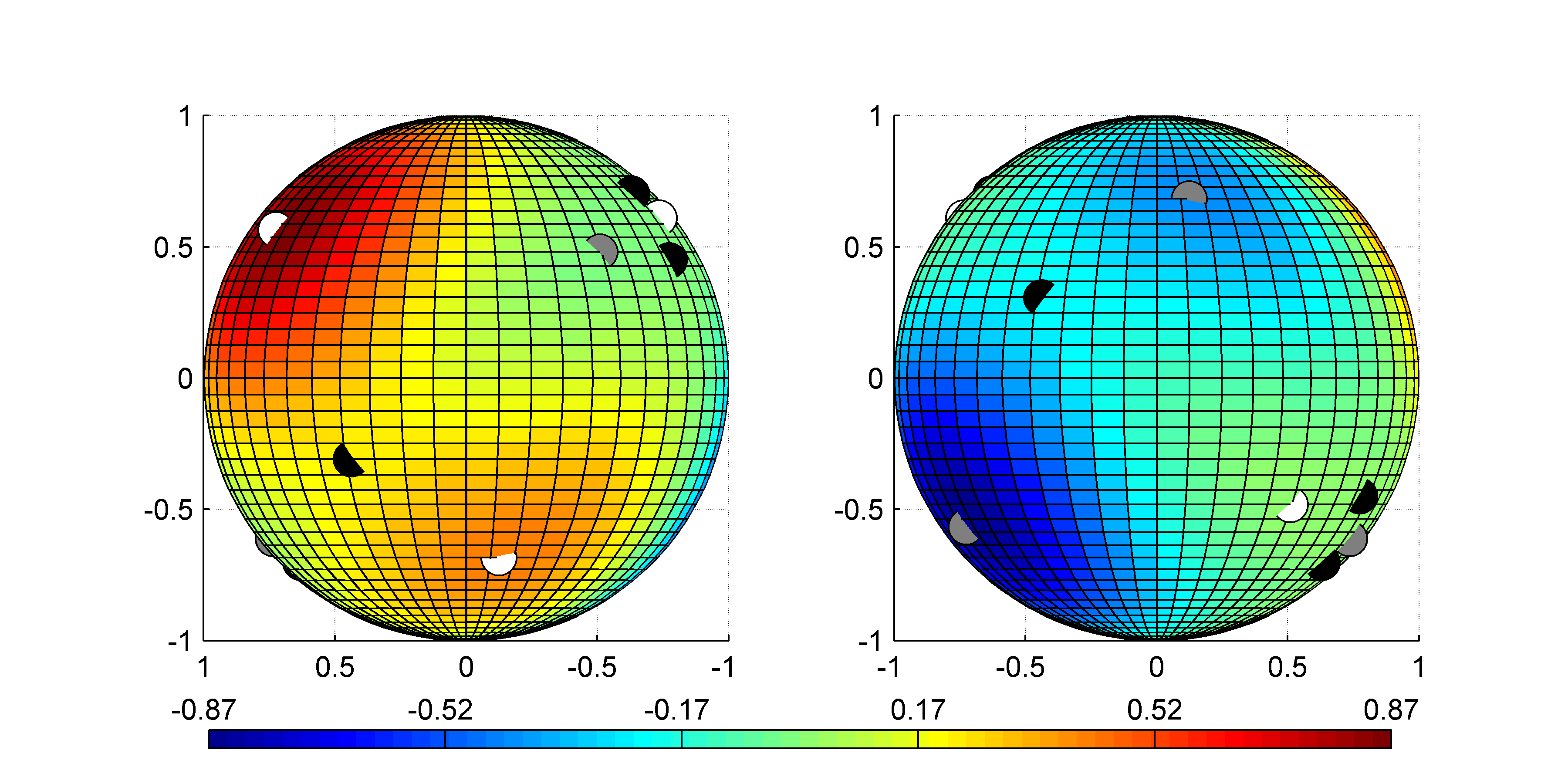}
  }

  \subfloat[SS-HOPM basins of attraction using $\alpha = 0$.]
  {
    \label{fig:sphere_odd_zero}
    \includegraphics[trim=35 0 35 25,clip,height=0.27\textheight]{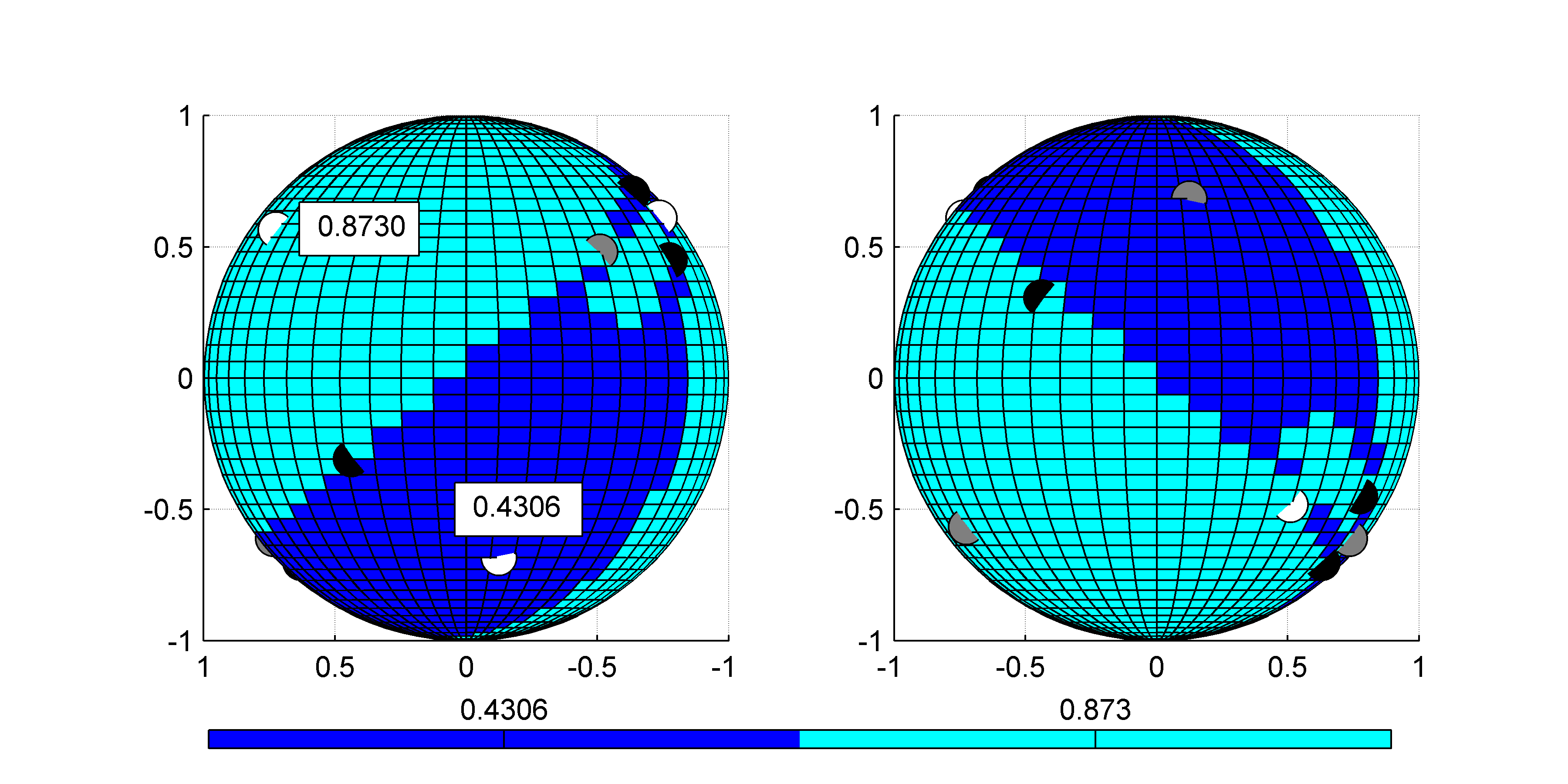}
  }

  \subfloat[SS-HOPM basins of attraction using $\alpha = 1$.]
  {
    \label{fig:sphere_odd_convex}
    \includegraphics[trim=35 0 35 25,clip,height=0.27\textheight]{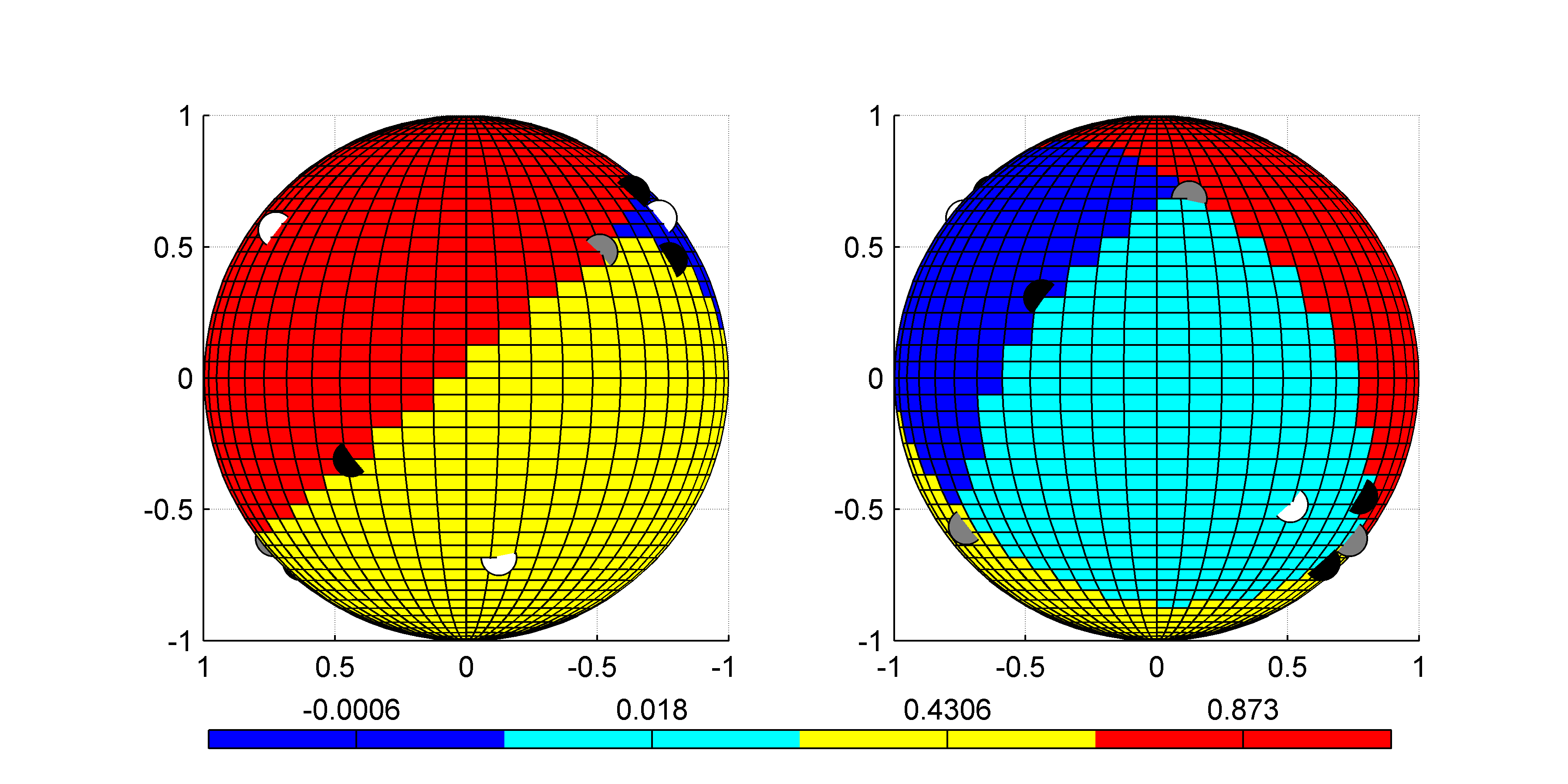}
  }

  \caption{Illustrations for $\TA \in \RT{3}{3}$ from \Ex{odd}. 
    The horizontal axis corresponds 
    to $\VE{x}{2}$ and the vertical axis to $\VE{x}{3}$; the left
    image is centered at $\VE{x}{1}=1$ and the right at
    $\VE{x}{1}=-1$.  White, gray, and black dots indicate eigenvectors
    that are negative stable, positive stable, and unstable,
    respectively. }
\end{figure}

\begin{example}
  We return again to $\TA \in \RT{3}{3}$ from \Ex{odd}, which is
  interesting because S-HOPM was able to find 2 of its eigenpairs
  without any shift.
  In \Fig{sphere_odd_f},
  $f(\Vx)$ is plotted on the unit sphere, along with each
  eigenvector, colored white, gray, or black based on whether it is
  negative stable, positive stable, or unstable, respectively. Observe that
  the function is antisymmetric, i.e., $f(\Vx) = -f(-\Vx)$.
  \Fig{sphere_odd_zero} shows the basins of attraction for S-HOPM
  (i.e., SS-HOPM with $\alpha = 0$). Every starting point converges to
  one of the 2 labeled eigenpairs. 
  This is not surprising because \Fig{jac_odd} shows that there are
  2 eigenvalues for which the spectral radius of the Jacobian is
  less than 1 ($\lambda = 0.8730$ and $0.4306$). The other 2
  eigenvalues are non-attracting for $\alpha=0$. \Fig{rate_odd} shows
  the observed rates of convergence.  
  
  \Fig{sphere_odd_convex} shows the basins of attraction for SS-HOPM
  with $\alpha = 1$; each negative stable eigenpair 
  (shown as a white circle) is an attracting
  eigenpair.  
  The concave case is just a mirror image and is not
  shown. 
\end{example}

As the previous example reminds us, for odd order, there is no need to
try both positive and negative $\alpha$ because the definiteness of
$C$ flips for eigenvectors of opposite sign.

Two additional examples of SS-HOPM are presented in \App{more}.

\subsection{Relationship to power method for matrix eigenpairs} 

The power method for matrix eigenpairs
is a technique for finding the largest-magnitude
eigenvalue (and corresponding eigenvector) of a diagonalizable symmetric matrix
\cite{GoVa96}.
Let $\M{A}$ be a symmetric real-valued $n \times n$ matrix. Then the matrix power
method is defined by
\begin{displaymath}
  \Vx_{k+1} = \frac{\M{A}\Vx_k}{\|\M{A}{\Vx_k}\|}.
\end{displaymath}
Assume that $\M{V} \M{\Lambda} \M{V}\Tra$ is the Schur
decomposition of $\M{A}$ with
eigenvalues satisfying $|\lambda_1| >
|\lambda_2| \geq \cdots \geq |\lambda_n|$ (note the strict difference
in the first 2 eigenvalues). The sequence $\{\Vx_k\}$ produced by
the matrix power method always converges (up to sign) to the
eigenvector associated with $\lambda_1$. Shifting the matrix by
$\M{A} \leftarrow \M{A} + \alpha \M{I}$ shifts the eigenvalues by
$\lambda_j \leftarrow \lambda_j + \alpha$, potentially altering which eigenvalue
has the largest magnitude.

In the matrix case, the eigenvalues of the Jacobian defined by \Eqn{J}
for an eigenpair
$(\lambda_j,\Vx_j)$ are given by
\begin{displaymath}
  \{0\} \cup \left\{ \frac{\lambda_i + \alpha}{\lambda_j + \alpha} : 1 \le i \le n \text{ with } i \ne j \right\}.
\end{displaymath}
Thus, the Jacobian at $\Vx_1$ is the only one such that $\rho(J(\Vx;\alpha))
< 1$; no other eigenvectors are stable fixed points of the iteration.
This corresponds to \Thm{fp} (or \Cor{fp}), since the most positive eigenvalue is negative stable,
the most negative eigenvalue is positive stable, and every other eigenvalue is unstable. The eigenpair
$(\lambda_1,\Vx_1)$ is an attractor for ordinary (convex) power iteration if $\lambda_1 > 0$ or for
flipped (concave) power iteration if $\lambda_1 < 0$.

In contrast to the matrix power method, SS-HOPM can find multiple
eigenpairs since there may be multiple positive and negative stable
eigenpairs. But, as for matrices, since the most positive and most negative
eigenvalues correspond to the global maximum
and minimum of $f(\Vx)$, they must be negative stable and positive stable
respectively. Thus, choosing $\alpha$
positive is necessary for finding the most positive tensor eigenvalue;
conversely, $\alpha$ negative is necessary for finding the most negative
tensor eigenvalue. Unfortunately, the ability to find multiple
eigenpairs means that there is no guarantee that the iterates will
converge to an extremal eigenpair from every starting
point. In fact, multiple starting points may be needed.

\subsection{Comparison to other methods}
SS-HOPM is useful for its guaranteed convergence properties and its
simple implementation based on tensor-vector multiplication.
For fixed $m$ and large $n$, the computational complexity of each iteration of
SS-HOPM is $O(n^m)$, which is the number of individual terms to be computed in $\TA\Vx^{m-1}$.
This is analogous to the $O(n^2)$ complexity of matrix-vector multiplication as used
in the matrix power method.
We do not yet know how the number of iterations needed for numerical convergence of SS-HOPM
depends on $m$ and $n$.

The convergence of SS-HOPM to only a subset of eigenvalues, which tend to be among
the largest in magnitude, is beneficial when the large eigenvalues are
of primary interest, as in the rank-1 approximation problem \cite{KoRe02}.
In particular, the most positive eigenvalue and most negative eigenvalue
always have a region of stable convergence for a suitable choice of shift.
However, the lack of stable convergence to certain other eigenvalues
is a disadvantage if those eigenvalues are of interest.

One evident computational approach for finding tensor eigenpairs should be compared with
SS-HOPM\@. This is to apply a numerical solver for nonlinear equation systems, such as Newton's
method, directly to the eigenvalue equations \Eqn{EVP}. The computational complexity
of each iteration of Newton's method for this system is that of SS-HOPM plus the construction
and inversion of the $(n + 1) \times (n + 1)$ Jacobian for $(\lambda, \Vx)$. The Jacobian construction
is effectively included in SS-HOPM, since it is dominated by computing $\TA\Vx^{m-2}$, which is
a precursor of $\TA\Vx^{m-1}$. The additional work for inversion is
$O(n^3)$, and for $m \ge 3$ it does not affect the complexity scaling, which remains
$O(n^m)$.

Two advantages of an approach such as Newton's method are generic locally stable convergence,
which enables finding eigenpairs not found by SS-HOPM, and the quadratic order of
convergence, which can be expected to require fewer iterations than the linearly
convergent SS-HOPM\@. On the other hand, there is no known guarantee of global
convergence as there is for SS-HOPM, and it is possible that many starting
points fail to converge. Even those that do converge may lead to eigenpairs of less
interest for a particular application. Furthermore, certain tensor structures can be more efficiently
handled with SS-HOPM than with Newton's method. For example, consider a higher-order symmetric
tensor expressed as a sum of terms, each of which is an outer product of matrices. The
computation of $\TA\Vx^{m-1}$ then reduces to a series of matrix-vector multiplications, which
are $O(n^2)$. This compares favorably to the $O(n^3)$ of Newton's method for the same tensor.
Further investigation of general nonlinear solver approaches to the tensor eigenvalue problem will
be beneficial.

Finally, we consider a polynomial solver approach, such as we implemented in Mathematica.
This can find all eigenpairs (subject to numerical conditioning issues) but becomes computationally
expensive for large $m$ and $n$. In part this is simply because, from \Thm{neigs}, the number of
eigenpairs grows exponentially with $n$. The solver in Mathematica is designed to find all solutions;
it is not clear whether a substantial improvement in efficiency would be possible if only one or a few
solutions were required.

Nevertheless, for comparison with the iterative approaches discussed above,
we have measured the computational time per eigenpair on a desktop computer for various values
of $m$ and $n$, as shown in \Fig{mathematica-timing}.
The complexity of the polynomial solution, even measured per eigenpair, is seen to increase extremely rapidly (faster than
exponentially) with $n$. Thus the polynomial solver approach is not expected to be practical for large $n$.

\begin{figure}[htbp]
  \centering
  \includegraphics[width=3in,trim=0 0 0 0,clip]{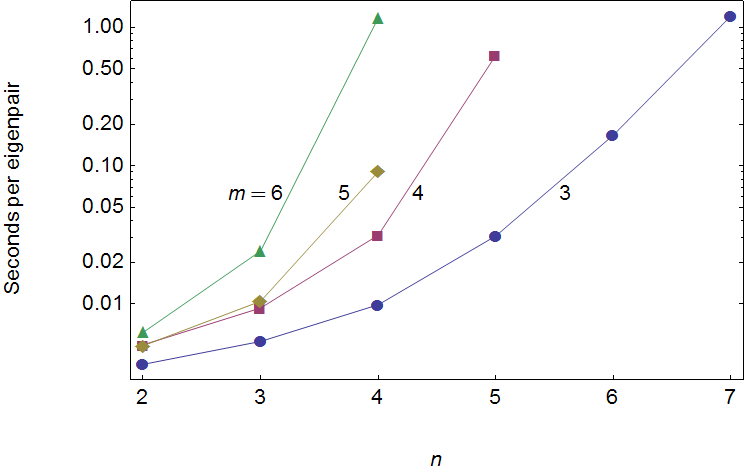}
  \caption{Average time (over 10 trials) required to compute all eigenpairs, divided by the number of eigenpairs,
  for random symmetric tensors in $\RT{m}{n}$. Note logarithmic vertical scale. Measured using \texttt{NSolve} in
  Mathematica on a 4 GHz Intel Core i7.}
  \label{fig:mathematica-timing}
\end{figure}

\section{Complex case}
\label{sec:complex}

We present the more general definition of complex eigenpairs and some related results, and propose an extension of the SS-HOPM algorithm to this case.

\subsection{Eigenrings}

Some of the solutions of the polynomial system that results from the eigenvalue equations may be complex; thus, the  definition can be extended to the complex case as follows, where $\Cct$ denotes the conjugate transpose.

\begin{definition}
  \label{def:EVP-COMPLEX}
  Assume that $\TA$ is a symmetric $m^\text{th}$-order $n$-dimensional
  real-valued tensor. 
  Then $\lambda \in \Complex$ is an
  \emph{eigenvalue} of $\TA$ if there exists $\Vx \in \Complex^n$ such that
  \begin{equation}\label{eq:EVP-COMPLEX}
    \TA\Vx^{m-1} = \lambda \Vx  \qtext{and} \Vx\Cct\Vx=1.
  \end{equation}%
  The vector $\Vx$ is a corresponding \emph{eigenvector}, and
  $(\lambda,\Vx)$ is called an \emph{eigenpair}. 
\end{definition}

\Def{EVP-COMPLEX} is closely related to the E-eigenpairs defined by Qi
\cite{Qi05,Qi07} but differs in the constraint on $\Vx$.\footnote{Qi
  \cite{Qi05,Qi07} requires $\Vx\Tra\Vx=1$ rather than
  $\Vx\Cct\Vx=1$.} It can also be considered as the obvious extension
of ($l^2$-)eigenpairs to $\Complex$.

It has been observed \cite{Qi07,CaSt10} that the complex eigenpairs of a tensor
form equivalence classes under a multiplicative transformation.
Specifically, if
$(\lambda,\Vx)$ is an eigenpair of $\TA \in \RT{m}{n}$ and $\Vy = e^{i\varphi}\Vx$
with $\varphi \in \Real$, 
then $\Vy\Cct\Vy = \Vx\Cct\Vx = 1$ and
\begin{displaymath}
  \TA\Vy^{m-1} 
  = e^{i(m-1)\varphi} \TA\Vx^{m-1}
  = e^{i(m-1)\varphi} \lambda \Vx
  = e^{i(m-2)\varphi} \lambda \Vy.
\end{displaymath}
Therefore $(e^{i(m-2)\varphi}\lambda,e^{i\varphi}\Vx)$ is also an
eigenpair of $\TA$ for any $\varphi \in \Real$. 
Consequently, if $\lambda$
is an eigenvalue, then any other $\lambda' \in \Complex$ with
$|\lambda'| = |\lambda|$ is also an 
eigenvalue. This leads to the notion of an eigenring.

\begin{definition}[Eigenring]
  For any $(\lambda,\Vx) \in \Complex \times \Complex^n$ that is an
  eigenpair of $\TA \in \RT{m}{n}$, we define a corresponding 
  equivalence class of (vector-normalized)  eigenpairs
  \begin{displaymath}
    \mathcal{P}(\lambda,\Vx) = 
    \{ (\lambda',\Vx') : 
    \lambda' = e^{i(m-2)\varphi} \lambda,
    \Vx' = e^{i\varphi} \Vx,
    \varphi \in \Real
    \},
  \end{displaymath}
  as well as a corresponding \emph{eigenring}
  \begin{displaymath}
      \mathcal{R}(\lambda) = \{ \lambda' \in \Complex : |\lambda'| =
      |\lambda| \}.
  \end{displaymath}
\end{definition}%
Thus, any eigenring contains either 1 or 2 real eigenvalues. The special case of real eigenpairs occurs whenever the corresponding eigenvector for one of those real eigenvalues is also real.

Since we assume that $\TA$ is real-valued, any nonreal eigenpairs
must come in sets of 2 related by complex conjugation, because taking the conjugate of
the eigenvalue equation does not change it. Such conjugate eigenpairs
are not members of the same equivalence class unless they are
equivalent to a real eigenpair.

An elegant result has recently been derived for the number of distinct
(non-equivalent) eigenvalues of a symmetric tensor.
The result was first derived for even-order tensors by Ni et al.\@ \cite{NiQiWaWa07} and later generalized by Cartwright and Sturmfels \cite{CaSt10} for all cases.
The case of $m=2$ requires
application of l'H\^opital's rule to see that there are $n$ eigenvalues.

\begin{theorem}[{Cartwright and Sturmfels \cite{CaSt10}}]
  \label{thm:neigs}
  A generic symmetric tensor $\TA\in\RT{m}{n}$ has
  $((m-1)^n-1)/(m-2)$ distinct eigenvalue equivalence classes.
\end{theorem}

These papers \cite{NiQiWaWa07,CaSt10} use the condition $\Vx\Tra\Vx = 1$ to
normalize eigenpairs, but in the generic case the result is the same for our
condition $\Vx\Cct\Vx = 1$.
This is because the eigenpairs with $\Vx\Cct\Vx = 1$ generically consist
of isolated equivalence classes that have
$\Vx\Tra\Vx \ne 0$ and thus can be rescaled to satisfy $\Vx\Tra\Vx = 1$,
giving a one-to-one correspondence between the distinct eigenvalues
in the two normalizations.
In special cases, however, the condition $\Vx\Cct\Vx = 1$ admits additional
eigenpairs with $\Vx\Tra\Vx = 0$.
Furthermore, tensors can be constructed with a continuous family of
non-equivalent eigenvectors that correspond to the same eigenvalue when
normalized by $\Vx\Tra\Vx$ but to infinitely many distinct eigenvalues
when normalized by $\Vx\Cct\Vx$ \cite[Example 5.7]{CaSt10}.

The polynomial system solver using Gr\"obner bases mentioned earlier
can also be used to find complex solutions.
A complication is that
our normalization condition $\Vx\Cct\Vx = 1$ is nonpolynomial due to the complex conjugation.
The system, however, becomes polynomial if the alternate normalization condition
$\Vx\Tra\Vx = 1$ is temporarily adopted.
Any such $\Vx$ can be rescaled to satisfy $\Vx\Cct\Vx = 1$.
Complex eigenvectors with $\Vx\Tra\Vx = 0$ will not be found,
but, as just noted, these do not occur generically.
Any nonreal solutions are transformed to a representative of the eigenring
with positive real $\lambda$ by setting
\begin{displaymath}
 (\lambda, \Vx) \leftarrow \left(\frac{|\lambda|}{(\Vx\Cct\Vx)^{m/2-1}},\, \left(\frac{|\lambda|}{\lambda}\right)^{1/(m-2)}\! \frac{\Vx}{(\Vx\Cct\Vx)^{1/2}}\right).
\end{displaymath}
This polynomial system solution becomes prohibitively expensive for large $m$ or $n$; however, for \Ex{KoRe02_ex1}, the nonreal eigenpairs can be computed this way and are presented in \Tab{KoRe02_ex1_nonreal}.
Thus, from this and \Tab{KoRe02_ex1}, we have found the 13 eigenvalue equivalence classes (real and nonreal) guaranteed by \Thm{neigs}.
\begin{table}
  \centering
  \footnotesize
  \begin{tabular}{|c|c|} \hline
$\lambda$ & $\Vx\Tra$ \\ \hline
 $\phantom{-}0.6694$ & [ $\phantom{-}0.2930 + 0.0571 i$ $\phantom{-}0.8171 - 0.0365 i$ $-0.4912 - 0.0245 i$ ] \\  \hline
 $\phantom{-}0.6694$ & [ $\phantom{-}0.2930 - 0.0571 i$ $\phantom{-}0.8171 + 0.0365 i$ $-0.4912 + 0.0245 i$ ] \\  \hline
  \end{tabular}
  \caption{Nonreal eigenpairs for $\T{A} \in \RT{4}{3}$ from \Ex{KoRe02_ex1}.}
  \label{tab:KoRe02_ex1_nonreal}
\end{table}

\subsection{SS-HOPM for Complex Eigenpairs}

We propose an extension of the SS-HOPM algorithm to the
case of complex vectors in \Alg{csshopm}.
Observe that the division by $\lambda_k + \alpha$ keeps the
phase of $\Vx_k$ from changing unintentionally.  It is akin to taking
the negative in the concave case in \Alg{sshopm}. It is important to
note that even if an eigenpair
is real, there is no guarantee that the complex SS-HOPM will converge
to the real eigenpair; instead, it will converge to some random
rotation in the complex plane. We have no
convergence theory in the convex case, but we present several
promising numerical examples.

\begin{algorithm}
  \caption{Complex SS-HOPM}
  \label{alg:csshopm}
    Given a tensor $\TA \in \RT{m}{n}$.
  \begin{algorithmic}[1]
    \Require $\Vx_0 \in \Complex^n$ with $\| \Vx_0 \| = 1$. Let
    $\lambda_0 = \TA \Vx_0^{m}$.
    \Require $\alpha \in \Complex$
    \For{$k=0,1,\dots$}
    \State $\hat \Vx_{k+1} \gets (\TA \Vx_k^{m-1} + \alpha \Vx_k) / 
    (\lambda_k + \alpha)$
    \State $\Vx_{k+1} \gets \hat \Vx_{k+1} / \| \hat \Vx_{k+1} \|$
    \State $\lambda_{k+1} \gets  \Vx_{k+1}\Cct\TA\Vx_{k+1}^{m-1}$
    \EndFor
  \end{algorithmic}
\end{algorithm}

\begin{example}
  We once again revisit $\TA \in \RT{4}{3}$ from \Ex{KoRe02_ex1} and test the complex version
  of SS-HOPM in \Alg{csshopm}.  \Tab{eigenrings_pos} shows the
  results of 100 runs using the same experimental conditions as in
  \Ex{KoRe02_ex1} except with complex random starting vectors.
  We find 7 distinct eigenrings --- the 6 stable real
  eigenpairs as well as a ring corresponding to the 2 complex
  eigenpairs. 
  \Fig{eigenrings_pos} shows the individual $\lambda_*$ values plotted
  on the complex plane. As mentioned above, it may converge anywhere
  on the eigenring, though there is clear bias toward the
  value of $\alpha$.

  To investigate this phenomenon further, we do another experiment with
  $\alpha = -(1+i)/\sqrt{2}$. It finds the same eigenrings as before
  as shown in \Tab{eigenrings_imag}, but this time the
  $\lambda_*$ values are distributed mostly in the lower left quadrant of the
  complex plane as shown in \Fig{eigenrings_imag}, again close to the value
  of $\alpha$.  In the case of
  the 2 complex eigenpairs with the same eigenring, the method
  finds the 2 distinct eigenvectors (i.e., defining 2
  different equivalence classes) in the 4 different times it
  converges to that eigenvalue; this is not
  surprising since the complex eigenvalue has 2 different
  eigenvectors as shown in \Tab{KoRe02_ex1}.

  We also ran an experiment with $\alpha=0$. In this case, 95 trials
  converged, but to non-eigenpairs (all with $|\lambda| = 0.3656$). In
  each case, even though $\{\lambda_k\}$ converged, we had
  $\|\Vx_{k+1} - \Vx_k\| \rightarrow 1.1993$, indicating that the
  sequence $\{\Vx_k\}$ had not converged and hence we did not obtain an $\Vx_*$
  satisfying the eigenvalue equation \Eqn{EVP-COMPLEX}. Although it is
  not shown, in all the convergent examples with the shifts mentioned
  above, the $\{\Vx_k\}$ sequence converged.
\end{example}

\begin{table}[htbp]
  \centering
  \caption{Eigenrings computed for $\TA \in \RT{4}{3}$ from \Ex{KoRe02_ex1} by complex SS-HOPM with 100 random starts.}
  \subfloat[$\alpha = 2$.]{
    \label{tab:eigenrings_pos}
    \footnotesize
    \begin{tabular}{|c|c|}\hline
\# Occurrences & $|\lambda|$ \\ \hline
  18 & 1.0954  \\ \hline 
  18 & 0.8893  \\ \hline 
  21 & 0.8169  \\ \hline 
   1 & 0.6694  \\ \hline 
  22 & 0.5629  \\ \hline 
   8 & 0.3633  \\ \hline 
  12 & 0.0451  \\ \hline
    \end{tabular}
  }
  ~~~
  \subfloat[$\alpha = \sqrt{2}(1+i)$ (2 failures).]{
    \label{tab:eigenrings_imag}
    \footnotesize
    \begin{tabular}{|c|c|}\hline
\# Occurrences & $|\lambda|$ \\ \hline
  22 & 1.0954  \\ \hline 
  15 & 0.8893  \\ \hline 
  12 & 0.8169  \\ \hline 
   4 & 0.6694  \\ \hline 
  16 & 0.5629  \\ \hline 
   9 & 0.3633  \\ \hline 
  20 & 0.0451  \\ \hline 
    \end{tabular}
  }
\end{table}

\begin{figure}[htbp]
  \centering
  \subfloat[$\alpha = 2$.]{
    \label{fig:eigenrings_pos}
    \includegraphics[width=0.33\textwidth,trim=20 5 20 0,clip]{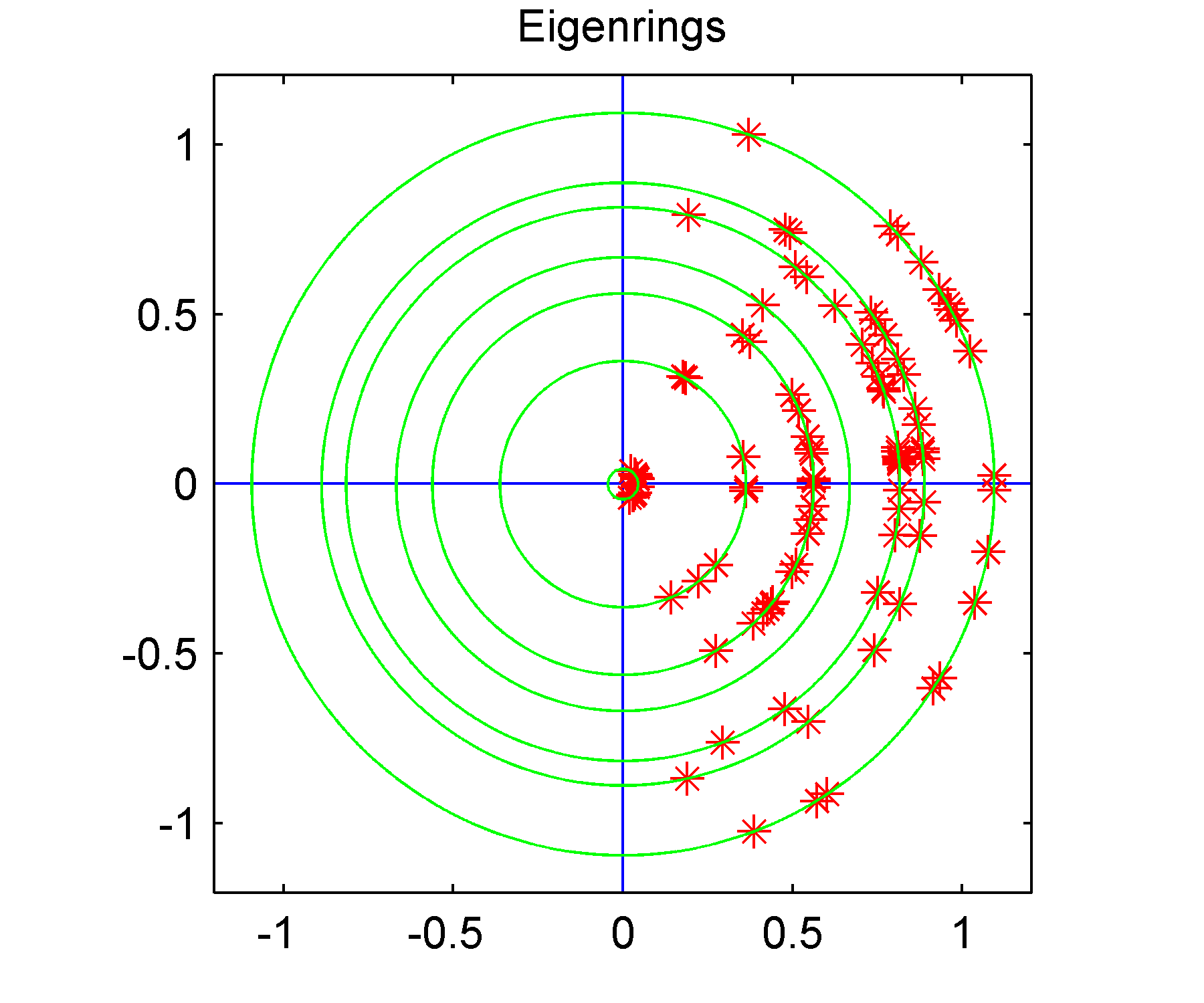}
  }
  \subfloat[$\alpha = -\sqrt{2}(1+i)$ (2 failures).]{
    \label{fig:eigenrings_imag}
    \includegraphics[width=0.33\textwidth,trim=20 5 20 0,clip]{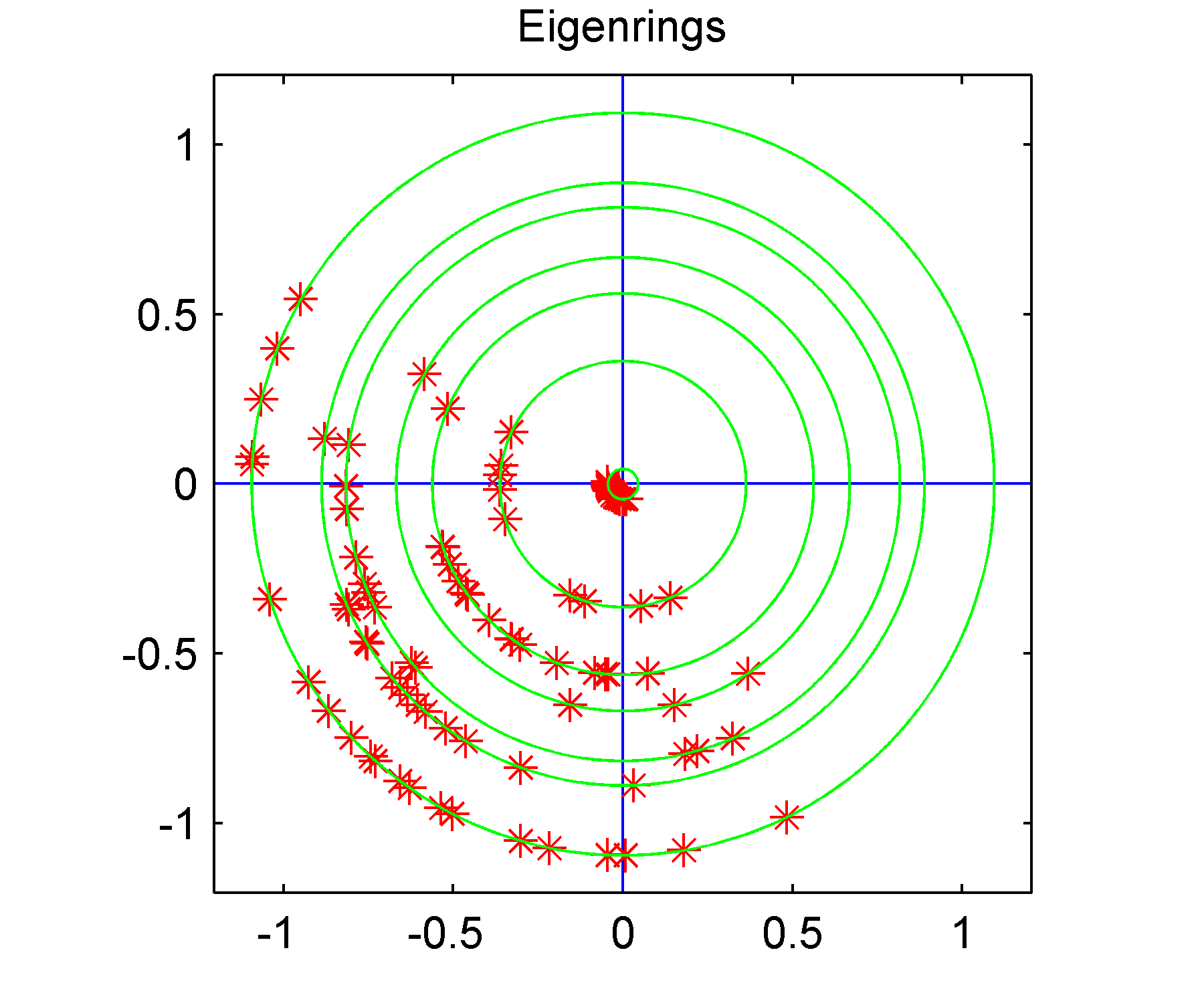}
  }
  \caption{For $\TA \in \RT{4}{3}$ from \Ex{KoRe02_ex1}, final $\lambda$ values (indicated by
  red asterisks) for 100 runs of complex SS-HOPM\@. The green lines
  denote the eigenrings.}
  \label{fig:complex}
\end{figure}

\section{Conclusions}
\label{sec:conclusions}
We have developed the shifted symmetric higher-order power method
(SS-HOPM) for finding tensor
eigenvalues. The method can be considered as a higher-order analogue
to the power method for matrices. Just as in the matrix case, it
cannot find all possible eigenvalues, but it is guaranteed to be able
to find the largest-magnitude eigenvalue. Unlike the matrix case, it
can find multiple eigenvalues; multiple starting points are typically needed to
find the largest eigenvalue.
A GPU-based implementation of SS-HOPM has been reported \cite{GPU}.

Building on previous work \cite{KoRe02,ReKo03,Er09}, we show that SS-HOPM will always converge to
a real eigenpair for an appropriate choice of $\alpha$. Moreover, using
fixed point analysis, we
characterize exactly which real eigenpairs can be found by the
method, i.e., those that are positive or negative stable. Alternative methods will need to be developed for finding the
unstable real eigenpairs, i.e., eigenpairs for which $C(\lambda,\Vx)$
is indefinite. A topic for future investigation is that the boundaries of the basins of
attraction for SS-HOPM seem to be defined by the non-attracting eigenvectors.

We present a complex version of SS-HOPM and limited experimental results
indicating that it finds eigenpairs, including complex
eigenpairs. Analysis of the complex version is a topic for future study.

Much is still unknown about tensor eigenpairs. For example, how do the
eigenpairs change with small perturbations of the tensor entries? 
Is there an eigendecomposition of a tensor? 
Can the convergence rate of the current method be accelerated? How
does one numerically compute unstable eigenpairs?
For computing efficiency, what is the optimal storage for
symmetric tensors? 
These are all potential topics of
future research.

\section*{Acknowledgments}
We thank Fangxiang Jiao and Chris Johnson (U. Utah), David Rogers
(Sandia), and Dennis Bernstein (U. Michigan) for contacting us with
interesting applications and providing test data. We thank Arindam
Banergee (U. Minnesota) for providing yet another motivating
application.
We thank Dustin Cartwright and Bernd Sturmfels (UC Berkeley) for
helpful discussions, 
especially about the number of eigenpairs for a problem of a given size.
We also thank our colleagues at Sandia for numerous helpful conversations
in the course of this work, especially Grey Ballad (intern from UC Berkeley).
%
We are grateful to the three anonymous referees for identifying
important references that we missed and for constructive
feedback that greatly improved the manuscript.

\opt{draft,siam}{
\bibliographystyle{siammod}
\bibliography{paper,allrefs}
}
\opt{arXiv}{

}

\clearpage
\appendix
\setlength{\floatsep}{0pt}
\section{Further examples}
\label{sec:more}
For additional insight, we consider two analytical examples.  In the
experiments in this section, each random trial used a point $\Vx_0$
chosen from a uniform distribution on $[-1,1]^n$. We allow up to 1000
iterations and say that the algorithm has converged if $|\lambda_{k+1}
- \lambda_k | < 10^{-15}$.

Consider the tensor $\TA\in\RT{3}{3}$ whose entries are 0 except where 
the indices are all unequal, in which case the entries are 1, i.e., 
\begin{equation}\label{eq:simple1}
  \TE{a}{ijk} =
  \begin{cases}
    1 & \text{if $(i,j,k)$ is some permutation of $(1,2,3)$}, \\
    0 & \text{otherwise}.
  \end{cases}
\end{equation}
Any eigenpair $(\lambda,\Vx)$ must satisfy the following system of equations:
\begin{align*}
2x_2 x_2 &= \lambda x_1, &
2x_1x_3 &= \lambda x_2, &
2 x_1 x_2 &= \lambda x_3, &
x_1^2 + x_2^2 + x_3^2 & = 1.
\end{align*}
The 7 real eigenpairs can be computed analytically and are listed
in \Tab{simple1}, from which we can see that there are 4 negative 
stable eigenpairs that should be identifiable
by SS-HOPM\@.  \Fig{jac_simple1} shows the spectral radius of the
Jacobian as $\alpha$ varies; the curve is identical for all 4
negative stable eigenpairs.

Another example is given as follows. Define the tensor $\T{A} \in \RT{4}{2}$ by
\begin{equation}
  \label{eq:simple2}
  a_{ijkl} = 0 \qtext{for all} i,j,k,l \qtext{except} a_{1111} = 1
  \qtext{and} a_{2222} = -1.
\end{equation}
The eigenpairs can be computed analytically as solutions to the
following system:
\begin{align*}
x_1^3 &= \lambda x_1, &
-x_2^3 &= \lambda x_2, &
x_1^2 + x_2^2 & = 1.
\end{align*}
The 2 real eigenpairs are listed in \Tab{simple2}, from which we can
see that one is negative stable and the other is positive stable.
\Fig{jac_simple1} shows the spectral radius of the Jacobian as
$\alpha$ varies. In this case, the spectral radius of the Jacobian can
be computed analytically; for $\lambda=1$, it is
$\frac{\alpha}{1+\alpha}$ and hence there is a singularity for
$\alpha=-1$.

\begin{table}[htbp]
  \centering
  \caption{Eigenpairs for two analytical problems.}
  \subfloat[Eigenpairs for $\T{A} \in \RT{3}{3}$ from \Eqn{simple1}.]{  
    \label{tab:simple1}
  \footnotesize
  \begin{tabular}{|c|c|c|c|c|} \hline
 $\lambda$ & $\Vx$ & eigenvalues of $C(\lambda,\Vx)$ & Type \\ \hline
 $0$  & [ $1$  $0$  $0$ ] &  $\{$  $-2$ ,  $2$  $\}$ & Unstable \\  \hline
 $0$  & [ $0$  $1$  $0$ ] &  $\{$  $-2$ ,  $2$  $\}$ & Unstable \\  \hline
 $0$  & [ $0$  $0$  $1$ ] &  $\{$  $-2$ ,  $2$  $\}$ & Unstable \\  \hline
 $2/\sqrt{3}$  & [ $\phantom{-}1/\sqrt{3}$  $\phantom{-}1/\sqrt{3}$ $\phantom{-}1/\sqrt{3}$ ] &  $\{$  $-2.3094$ ,  $-2.3094$  $\}$ & Neg. Stable \\  \hline
 $2/\sqrt{3}$  & [ $\phantom{-}1/\sqrt{3}$  $-1/\sqrt{3}$  $-1/\sqrt{3}$ ] &  $\{$  $-2.3094$ ,  $-2.3094$  $\}$  & Neg. Stable \\  \hline
 $2/\sqrt{3}$  & [ $-1/\sqrt{3}$  $\phantom{-}1/\sqrt{3}$  $-1/\sqrt{3}$ ] &  $\{$  $-2.3094$ ,  $-2.3094$  $\}$  & Neg. Stable \\  \hline
 $2/\sqrt{3}$  & [ $-1/\sqrt{3}$  $-1/\sqrt{3}$  $\phantom{-}1/\sqrt{3}$ ] &  $\{$  $-2.3094$ ,  $-2.3094$  $\}$  & Neg. Stable \\  \hline
  \end{tabular}
  }\\
  \subfloat[Eigenpairs for $\T{A} \in \RT{4}{2}$ from \Eqn{simple2}.]{
  \label{tab:simple2}
  \footnotesize
  \begin{tabular}{|c|c|c|c|c|} \hline
 $\lambda$ & $\Vx$ & eigenvalues of $C(\lambda,\Vx)$ & Type \\ \hline
 $1$  & [ $1$  $0$ ] &  $\{$  $-1$ $\}$ & Neg. Stable \\  \hline
 $-1$  & [ $0$  $1$ ] &  $\{$  $1$ $\}$ & Pos. Stable \\  \hline
  \end{tabular}
  }
\end{table}

\begin{figure}[htbp]
  \centering
  \subfloat[\Eqn{simple1} (All four are identical.) ]{
    \label{fig:jac_simple1}
  \includegraphics[width=2.4in,trim=0 0 0 0,clip]{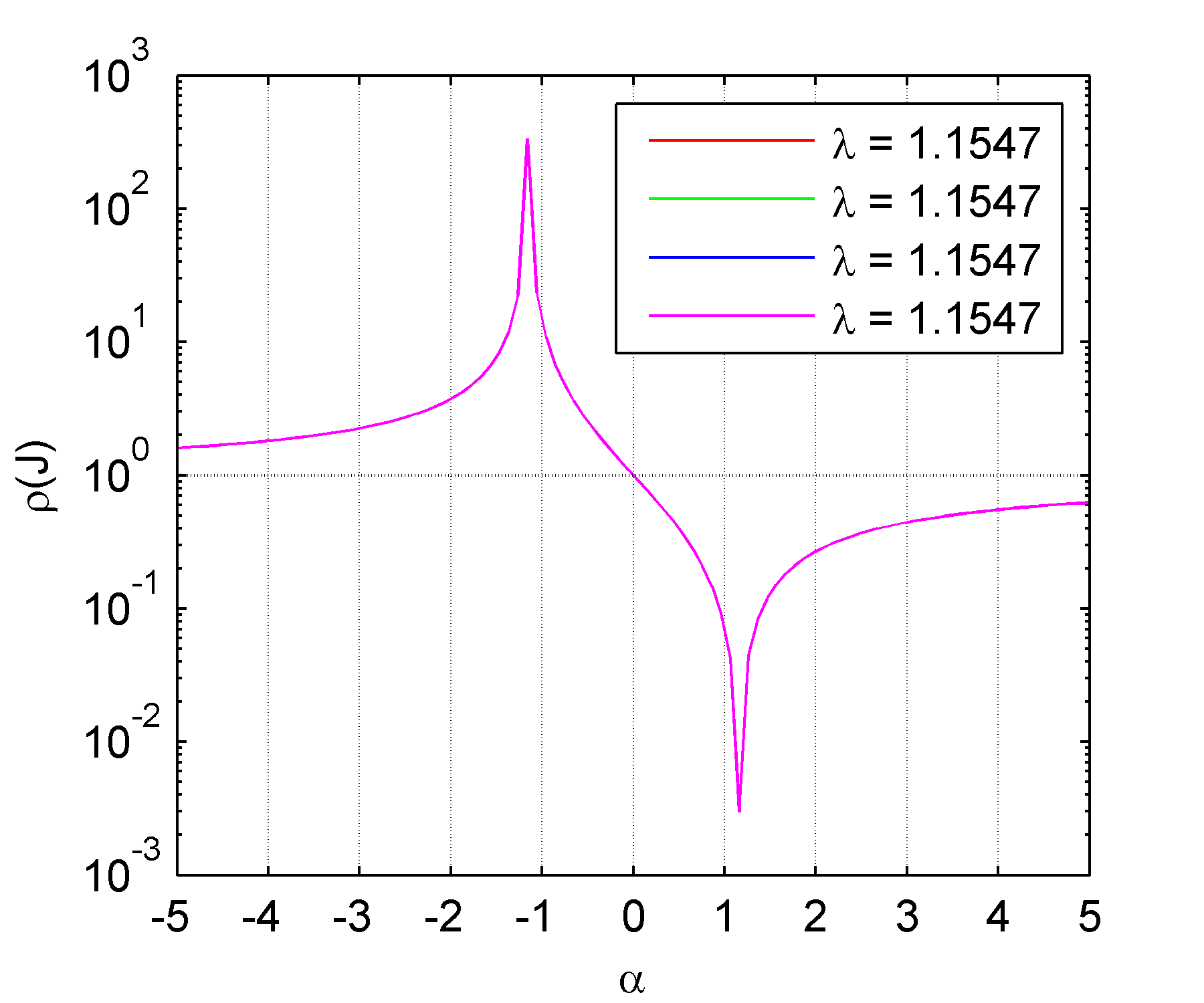}
  }
  \subfloat[\Eqn{simple2}]{  
    \label{fig:jac_simple2}
    \includegraphics[width=2.4in,trim=0 0 0 0,clip]{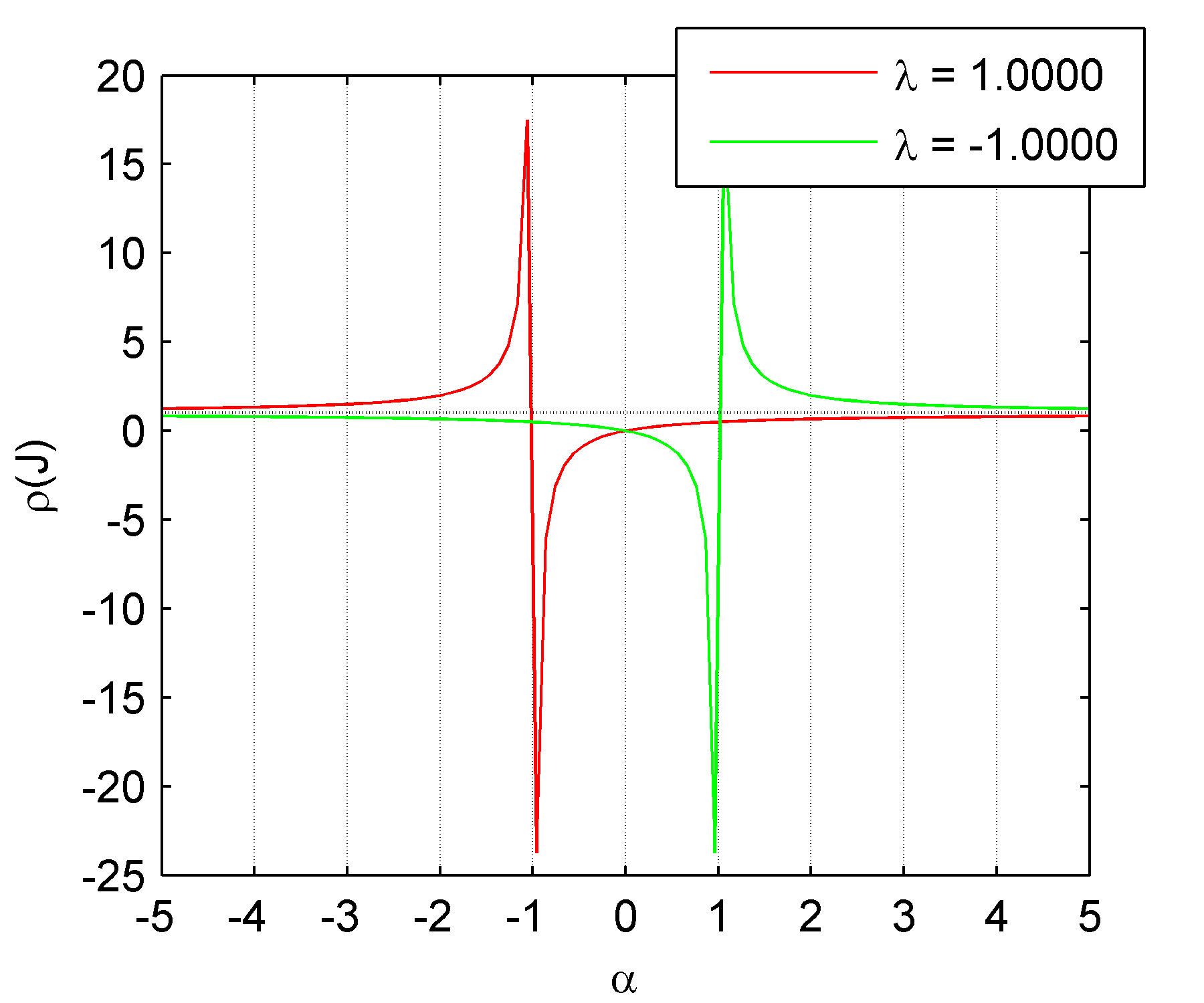} 
  }

  \caption{Spectral radii of the Jacobian $J(\Vx;\alpha)$ for
    different eigenpairs as $\alpha$ varies. }
\end{figure}

For \Eqn{simple1}, we ran 100 trials with $\alpha=0$, and none
converged, as expected per \Fig{jac_simple1}.  The results of
100 random trials with $\alpha=12$ (the ``conservative choice'') is
shown in \Tab{simple1-12}, in which case every trial converged to one
of the 4 negative stable eigenpairs. (Note that $2/\sqrt{3} \approx
1.1547$ and $1/\sqrt{3} \approx 0.5774$.)
\Tab{simple1-1} shows the results of 100 random trials with
$\alpha=1$. As expected (per \Fig{jac_simple1}), the convergence is
much faster.
\begin{table}[htbp]
  \centering
  \caption{Eigenpairs for $\TA\in\RT{3}{3}$  from \Eqn{simple1} 
    computed by SS-HOPM\@.}
  \label{tab:simple1-results}
  \subfloat[$\alpha=12$]{
    \label{tab:simple1-12}
    \footnotesize
    \begin{tabular}{|c|c|c|c|}
      \hline
      \# Occurrences & $\lambda$ & $\Vx$ & Median Its. \\ \hline
      22 &  $\phantom{-}1.1547$  & [ $-0.5774$  $\phantom{-}0.5774$  $-0.5774$ ] &   92 \\ \hline 
      18 &  $\phantom{-}1.1547$  & [ $\phantom{-}0.5774$  $\phantom{-}0.5774$  $\phantom{-}0.5774$ ] &   90 \\ \hline 
      29 &  $\phantom{-}1.1547$  & [ $-0.5774$  $-0.5774$  $\phantom{-}0.5774$ ] &   91 \\ \hline 
      31 &  $\phantom{-}1.1547$  & [ $\phantom{-}0.5774$  $-0.5774$  $-0.5774$ ] &   94 \\ \hline 
    \end{tabular}
  }\\
  \subfloat[$\alpha=1$]{
    \label{tab:simple1-1}
    \footnotesize
    \begin{tabular}{|c|c|c|c|}
      \hline
      \# Occurrences & $\lambda$ & $\Vx$ & Median Its. \\ \hline
      22 &  $\phantom{-}1.1547$  & [ $\phantom{-}0.5774$  $-0.5774$  $-0.5774$ ] &    9 \\ \hline 
      25 &  $\phantom{-}1.1547$  & [ $-0.5774$  $\phantom{-}0.5774$  $-0.5774$ ] &    9 \\ \hline 
      26 &  $\phantom{-}1.1547$  & [ $\phantom{-}0.5774$  $\phantom{-}0.5774$  $\phantom{-}0.5774$ ] &    9 \\ \hline 
      27 &  $\phantom{-}1.1547$  & [ $-0.5774$  $-0.5774$  $\phantom{-}0.5774$ ] &    9 \\ \hline 
    \end{tabular}
  }  
\end{table}
For \Eqn{simple2}, we ran 100 trials with $\alpha=0.5$
(\Tab{simple2-p}) and 100 trials with $\alpha=-0.5$
(\Tab{simple2-n}). We find the negative stable and positive stable
eigenvalues as expected.

\begin{table}[h]
  \centering
  \caption{Eigenpairs for $\TA\in\RT{4}{2}$  from \Eqn{simple2} 
    computed by SS-HOPM\@.}
  \label{tab:simple2-results}
  \subfloat[$\alpha=0.5$]{
    \label{tab:simple2-p}
    \footnotesize
    \begin{tabular}{|c|c|c|c|}
      \hline
\# Occurrences & $\lambda$ & $\Vx$ & Median Its. \\ \hline
 100 &  $\phantom{-}1.0000$  & [ $-1.0000$  $\phantom{-}0.0000$ ] &   16 \\ \hline 
    \end{tabular}
  }\\
  \subfloat[$\alpha=-0.5$]{
    \label{tab:simple2-n}
    \footnotesize
    \begin{tabular}{|c|c|c|c|}
      \hline
\# Occurrences & $\lambda$ & $\Vx$ & Median Its. \\ \hline
 100 &  $-1.0000$  & [ $\phantom{-}-0.0000$  $\phantom{-}1.0000$ ] &   16 \\ \hline 
    \end{tabular}
  }  
\end{table}

\end{document}